\documentclass[12pt]{amsart}
\usepackage{amssymb}
\usepackage{amsmath}
\usepackage{mathtools}
\usepackage{calrsfs}
\usepackage[all]{xy}
\usepackage{url}
\newtheorem{theorem}{Theorem}[section]
\newtheorem{lemma}[theorem]{Lemma}
\newtheorem{proposition}[theorem]{Proposition}
\newtheorem{corollary}[theorem]{Corollary}

\theoremstyle{definition}
\newtheorem{definition}[theorem]{Definition}
\newtheorem{example}[theorem]{Example}
\newtheorem{remark}[theorem]{Remark}

\theoremstyle{remark}

\newcommand{\ZZ}{\mathbb{Z}}
\newcommand{\QQ}{\mathbb{Q}}

\newcommand{\CC}{\mathbb{C}}

\newcommand{\bA}{\mathbf{A}}

\DeclareMathOperator{\Log}{Log}

\DeclareMathOperator{\GL}{GL}
\DeclareMathOperator{\Mat}{Mat}

\DeclareMathOperator{\Frob}{Frob}

\DeclareMathOperator{\Pic}{Pic}

\DeclareMathOperator{\Gal}{Gal}
\DeclareMathOperator{\Hom}{Hom}
\DeclareMathOperator{\rank}{rank}

\DeclareMathOperator{\Exp}{Exp}
\DeclareMathOperator{\Id}{Id}

\DeclareMathOperator{\Lie}{Lie}

\DeclareMathOperator{\ad}{ad}

\newcommand{\sep}{\mathrm{sep}}

\newcommand{\tr}{\mathrm{tr}}

\newcommand{\inorm}[1]{{\lvert #1 \rvert}_{\infty}}

\begin{document}
	
	\title[Special values of Goss $L$-series attached to Drinfeld modules of rank 2]{Special values of Goss $L$-series attached to Drinfeld modules of rank 2}
	
	\author{O\u{g}uz Gezm\.{i}\c{s}}
	\address{Department of Mathematics, National Tsing Hua University, Hsinchu City 30042, Taiwan R.O.C.}
	\email{gezmis@math.nthu.edu.tw}

	\thanks{This project was  supported by MoST Grant 109-2811-M-007-553.}
	

	\begin{abstract}
		Inspired by the classical setting, Goss defined $L$-series attached to Drinfeld modules. In this paper, for a fixed choice of a power $q$ of a prime number and a given Drinfeld module $\phi$ of rank 2 with a certain condition on its coefficients, we give explicit formulas for the values of Goss $L$-series attached to $\phi$ at positive integers $n$ such that $2n+1\leq q$ in terms of polylogarithms and coefficients of the logarithm series of $\phi$.
	\end{abstract}
	
	
	\maketitle
	\section{Introduction}
	\subsection{Background and Motivation}
	One of the major sources of constructing $L$-functions is related to Galois representations. For a number field $F$, let $\bar{F}$ be its algebraic closure and $G_F$ be its absolute Galois group. Let $l$ be a prime number. Consider a collection $\rho:=(\rho_l)$ of $l$-adic representations which are  continuous homomorphisms $\rho_l:G_F\to \GL(V_l)$ where $V_l$ is a finite dimensional $\QQ_l$-vector space. Then we say $\rho$ forms a strictly compatible system if there exists a finite set $U$ of places of $F$ such that 
	\begin{itemize}
		\item[(i)] For all $\mu\not \in U$ and all $l$ relatively prime to $\mu$, $\rho_l$ is unramified at $\mu$.
		\item[(ii)] For such $\mu$ and $l$, the polynomial 
		\[
		P_\mu(X):=\det(1-X\rho_l(\Frob_\mu) \ \ |V_l)
		\]
		has coefficients in $\QQ$ and is independent of $l$.
	\end{itemize}
	For example, let $E$ be an abelian variety of dimension $g$ over $F$ and for any $i\in \mathbb{Z}_{\geq 0}$, $E[l^i]$ be the group of all $l^i$-torsion points of $E$ in $\bar{F}$. Then one can consider $V_l$ as the following vector space
	\[
	V_l:=\lim_{\substack{\leftarrow\\i}}E[l^i]\otimes_{\mathbb{Z}_l}\QQ_l\cong \QQ_l^{2g}
	\]
	to see that the family $\rho=(\rho_l)$ of representations $\rho_l:G_F\to \GL_{2g}(V_l)$ induced from the continuous action of $G_F$ on $E[l^i]$ forms a strictly compatible system (see \cite{Del} and \cite{LP95} for details).
	
	For each such system $\rho=(\rho_l)$ of Galois representations, one can assign the $L$-function
	\[
	L_{U}(\rho,s):=\prod_{\mu}P_\mu(\mathcal{N}\mu^{-s})^{-1}
	\]
	where $\mu$ runs over all finite places not in $U$ and $\mathcal{N}\mu$ is the norm of the place $\mu$ of $F$. The function $L_{U}(\rho,s)$ converges to an analytic function for $s\in \CC$ when the real part $\Re(s)$ of $s$ is sufficiently large. 
	We refer the reader to \cite{Del} and \cite{Ser} for further details about the subject.

	\subsection{Drinfeld $A$-modules and $L$-series}
	In the present paper, we focus on the special values of an analogue of aforementioned $L$-series in the positive characteristic case whose construction is due to Goss \cite{Goss2}. Let $\mathbb{F}_q$ be the finite field with $q$ elements and $\theta$ be an independent variable over $\mathbb{F}_q$. We define $A$ to be the set of polynomials in $\theta$ with coefficients in $\mathbb{F}_q$ and $A_{+}$ to be the set of monic polynomials in $A$. Let $K$  be the fraction field of $A$ and $K_{\infty}$ be the completion of $K$ at the infinite place with respect to the norm $|\cdot|_{\infty}$ normalized so that $\inorm{\theta}=q$. We also set $\CC_{\infty}$ to be the completion of a fixed algebraic closure of $K_{\infty}$. 
	
	Let $K^{\text{sep}}$ be the separable closure of $K$ in $\CC_{\infty}$. Let $t$ be another independent variable and set $\bA:=\mathbb{F}_q[t]$. For any monic irreducible polynomial $w$ of $\bA$, we define $\mathbf{K}_w$ to be the completion of $\mathbb{F}_q(t)$ at $w$. Consider a family $\rho=(\rho_{w})$ of continuous representations of $\Gal(K^{\sep}/K)$ on a finite dimensional $\mathbf{K}_w$-vector space $V_w$. For any prime element $v$ of $A_{+}$, let us set $\Frob_v$ to be the geometric Frobenius at $v$. We say $\rho$ forms a strictly compatible system if there is a finite set $U'$ of primes of $A_{+}$ such that 
	\begin{itemize}
		\item[(i)] For all $v\not \in U'$ and all $w_{|t=\theta}\in A$ relatively prime to $v$, $\rho_{w}$ is unramified at $v$.
		\item[(ii)] For such $v$ and $w$, the polynomial
		\[
		P_v(X):=\det(1-X\rho_w(\Frob_v) \ \ |V_w)
		\]
		has coefficients in $\mathbb{F}_q(t)$ and is independent of $w$. 
	\end{itemize}
	Analogously, we define the $L$-function $L_{U'}(\rho,n)$ corresponding to a strictly compatible system $\rho=(\rho_w)$ of representations  $\rho_w:\Gal(K^{\sep}/K)\to \GL(V_w)$ by
	\begin{equation}\label{E:lfunct}
	L_{U'}(\rho,n):=\prod_{v\not \in U'}P_{v}(v_{|\theta=t}^{-n})^{-1}
	\end{equation}
	where $v$ runs over prime elements of $A_{+}$ not in $U'$.
	For integer values of $n$, the function $L_{U'}(\rho,n)$ converges in the Laurent series ring $\mathbb{F}_q((1/t))$ when $n$ is sufficiently large.  For further details, we refer the reader to \cite{BP09} and  \cite{TagWan}. 
	
	We point out that throughout the paper, by a slight abuse of notation, we continue to denote the value constructed by setting $t=\theta$ in $L_{U'}(\rho,n)$ by the same notation and hence our $L$-values will converge in $K_{\infty}$.

	Let $L$ be a field extension of $K$ in $\CC_{\infty}$. We define the twisted power series ring $L[[\tau]]$ with the rule $\tau c=c^q\tau$ for all $c\in L$ and set $L[\tau]$ to be the subring of $L[[\tau]]$ containing only polynomials in $\tau$.
	
	A Drinfeld $A$-module $\phi$ of rank $r$ is an $\mathbb{F}_q$-linear ring homomorphism $\phi:A\to L[\tau]$ defined by 
	\begin{equation}\label{D:drinfeld}
	\phi_{\theta}=A_0+A_1\tau+\dots+A_r\tau^r
	\end{equation}
	so that $A_0=\theta$ and $A_r\neq 0$. For each $0\leq i \leq r$, we call $A_i$ the $i$-th coefficient of $\phi$.
	
	One can assign the exponential series $\exp_{\phi}=\sum_{i\geq 0} \xi_i\tau^i\in L[[\tau]]$ to  $\phi$ subject to the condition that $\xi_0=1$ and $\exp_{\phi}\theta=\phi_{\theta}\exp_{\phi}$. The logarithm series of $\phi$ 
	\[
	\log_{\phi}=\sum_{i\geq 0}\gamma_{i}\tau^i\in L[[\tau]]
	\]
	is defined with respect to the condition that $\gamma_0=1$ and $\theta \log_{\phi}=\log_{\phi}\phi_{\theta}$. It is also the formal inverse of the exponential series $\exp_{\phi}$ in $L[[\tau]]$. 
	
	One of the examples of Drinfeld $A$-modules is the Carlitz module $C$ given by 
	\[
	C_{\theta}=\theta+\tau
	\]
	and its relation with the class field theory has been studied by Carlitz \cite{C35}, \cite{C38} and Hayes \cite{H74}. 
	
	For any non-negative integer $n$, we set $[n]:=1$ if $n=0$ and $[n]:=\theta^{q^n}-\theta$ otherwise. The exponential series  $\exp_{C}$ of the Carlitz module is defined by 
	\[
	\exp_{C}=\sum_{i\geq 0}\frac{\tau^i}{D_i}\in K[[\tau]]
	\]
	where $D_0:=1$ and $D_i:=[i]D_{i-1}^q$ for $i\geq 1$. Furthermore the logarithm series $\log_{C}$ can be given by 
	\[
	\log_{C}=\sum_{i\geq 0}\frac{\tau^i}{L_i}\in K[[\tau]]
	\]
	where $L_0:=1$ and $L_i=(-1)^{i}[i][i-1]\dots[1]$ when $i\geq 1$. Using the coefficients of $\log_{C}$, one can also define the $n$-th polylogarithm function $\log_{n}$ given by
	\[
	\log_{n}(z)=\sum_{i\geq 0}\frac{z^{q^i}}{L_i^{n}}
	\]
	whenever $\inorm{z}<nq/(q-1)$. For more information on Drinfeld $A$-modules, we refer the reader to \cite[Sec. 3 and 4]{Goss} and \cite[Sec. 2 and 3]{Thakur}.
	
	Taelman \cite{Taelman2} introduced effective $t$-motives which can be seen as a generalization of Anderson $t$-motives defined in \cite{And86}. Using Anderson's theory \cite{And86}, one can show that for every Drinfeld $A$-module $\phi$, there exists a unique corresponding effective $t$-motive $M_{\phi}$ up to isomorphism. Furthermore in \cite{G01}, Gardeyn proved that one can construct a strictly compatible system of Galois representations $\rho=(\rho_{w})$ attached to $M_{\phi}$ with a certain choice of the $\textbf{K}_w$-vector space $V_w$ (see \S 2.3 for definitions and details). 
	
	Let $n$ be a positive integer and $L_{U'}(M_{\phi},n)$ be the value of the $L$-function defined as in \eqref{E:lfunct} corresponding to the system of Galois representations $\rho=(\rho_{w})$ attached to $M_{\phi}$.  Our main purpose in the present paper is to study certain special values of the $L$-function $L(M_{\phi},n):=L_{U'}(M_{\phi},n)$ when $\phi$ is a Drinfeld $A$-module of rank 2  defined under some conditions on its coefficients. To motivate our main result, we first explain the well-known Carlitz module case: Consider the $q$-expansion of $n$ given by $n=\sum n_jq^j$ where $0\leq n_j \leq q-1$ and $n_j=0$ for $j\gg 0$. Set $\Gamma_{n+1}:=\prod_{j\geq 0} D_j^{n_j}\in A$. Thanks to the results of Hsia and Yu \cite[Thm. 3.1]{HY} and Anderson and Thakur \cite[Thm. 3.8.3]{AndThak90}, we know that 
	\[
	L(M_C,n+1)=\sum_{a\in A_{+}}\frac{1}{a^{n}}=\frac{1}{\Gamma_{n}}\sum_{i=0}^{m}h_{j}\log_{n}(\theta^j)
	\]
	for some $h_j\in A$ and $m<nq/(q-1)$.

	\subsection{The Main Result} 
	Let $\phi$ be a Drinfeld $A$-module of rank 2 defined by
	\begin{equation}\label{E:rank22}
	\phi_{\theta}=\theta+a\tau+b\tau^2
	\end{equation}
	where $a\in L$ and $b\in L\setminus\{0\}$. For any finite set $S\subset\mathbb{Z}_{\geq 0}$ and a non-negative integer $j$, let us define $S+j:=\{i+j:i\in S\}$. Set $\mathcal{P}_2(0):=\{(\emptyset,\emptyset)\}$ and for any $n\geq 1$, we define
	\[
	\mathcal{P}_2(n):=\{(S_1,S_2):S_1\cap S_2=\emptyset \text{ and } S_i\subset\{0,1,\dots,n-1\},\ \  i=1,2\}
	\]
	to be the set of tuples $(S_1,S_2)$ such that $S_1$, $S_2$ and $S_2+1$ are distinct and form a partition of $\{0,1,\dots,n-1\}$. We call the elements of $\mathcal{P}_2(n)$ shadowed partitions. For any positive integer $n$, we also set $\mathcal{P}_2^{1}(n)$ to be the set of shadowed partitions $(S_1,S_2)\in \mathcal{P}_2(n)$ such that $0\in S_1$.
	
	Define $w_1(S):=0$ if $S=\emptyset$ and $w_1(S):=\sum_{i\in S}q^i$ otherwise. Furthermore, for any finite set $S\subset \mathbb{Z}_{\geq 0}$ with $0\in S$, we let  $w_2(S):=0$ if $S=\{0\}$ and $w_2(S):=\sum_{i\in S\setminus \{0\}}q^i$ if $\{0\}\subsetneq S$. 
	For any $\mathcal{U}=(S_1,S_2)\in \mathcal{P}_2(n)$, we define the component $\mathcal{C}_{\mathcal{U}}$ of $\gamma_n$ corresponding to $\mathcal{U}$ by
	\[
	\mathcal{C}_{\mathcal{U}}:=\begin{cases}\frac{a^{w_1(S_1)}b^{w_1(S_2)}}{\prod_{i\in S_1}(-[i+1])\prod_{i\in S_2}(-[i+2])}& \text{ if } \mathcal{U}=(S_1,S_2)\in \mathcal{P}_2(n)\setminus \mathcal{P}_2^{1}(n)\\
	\frac{a^{w_2(S_1)}b^{w_1(S_2)}}{\prod_{i\in S_1}(-[i+1])\prod_{i\in S_2}(-[i+2])}& \text{ if } \mathcal{U}=(S_1,S_2)\in \mathcal{P}_2^{1}(n).
	\end{cases}
	\] 
	We set $F_0:=0$ and  $T_0:=1$. For $n\geq 1$, define 
	\[
	F_n:=\sum_{\mathcal{U}\in \mathcal{P}_2^{1}(n)}\mathcal{C}_{\mathcal{U}}
	\]
	and 
	\[
	T_n:=\sum_{\mathcal{U}\in \mathcal{P}_2(n)\setminus \mathcal{P}_2^{1}(n)}\mathcal{C}_{\mathcal{U}}.
	\]
	El-Guindy and Papanikolas \cite[Thm. 3.3]{EP13} showed that for any $n\geq 0$, the $n$-th coefficient $\gamma_{n}$ of the logarithm series $\log_{\phi}$ can be given by
	\[
	\gamma_n=aF_n+T_n.
	\]
	
	We now further assume that $\phi$ is the Drinfeld $A$-module as in \eqref{E:rank22} such that $a\in \mathbb{F}_q$ and $b\in \mathbb{F}_q^{\times}$. Let $\tilde{\phi}$ be another Drinfeld $A$-module given by $\tilde{\phi}_{\theta}=(-b^{-1})^{-1/(q-1)}\phi_{\theta}(-b^{-1})^{1/(q-1)}=\theta -ab^{-1}\tau+b^{-1}\tau^2$ for a fixed $(q-1)$-st root of $-b^{-1}$. The class number formula \cite[Thm. 1]{Taelman} of Taelman yields that 
	\[
	L(M_{\phi},1)=\log_{\tilde{\phi}}(1)
	\]
	where $\log_{\tilde{\phi}}$ is the logarithm function of $\tilde{\phi}$ induced by its logarithm series (see Remark \ref{R:valatone} for details).
	
	Our main result which concerns special values of $L(M_{\phi},n)$ at integers $n\geq 2$ in a certain domain can be stated as follows (later stated as Theorem \ref{T:specvalue}).
	\begin{theorem}\label{T:intr} Let $\phi$ be a Drinfeld $A$-module of rank 2 given by 
		\[
		\phi_{\theta}=\theta +a\tau+b\tau^2
		\]
		such that $a\in \mathbb{F}_q$ and $b\in \mathbb{F}_q^{\times}$. Then for any positive integer $n$ satisfying $ 2n+1 \leq q$, we have
		\begin{multline*}	
		L(M_\phi,n+1)
		=\bigg(\sum_{i=0}^{\infty}\frac{(-1)^{i}b^{-i}\gamma_i}{L_i^{n}}\bigg)\bigg(1+\sum_{i=1}^{\infty}\frac{(-1)^{i}b^{-(i-1)}F_{i-1}}{L_i^{n}}\bigg)\\
		-\bigg(\sum_{i=1}^{\infty}\frac{(-1)^{i}b^{-(i-1)}\gamma_{i-1}}{L_i^{n}}\bigg)\bigg(\sum_{i=0}^{\infty}\frac{(-1)^{i}b^{-i}F_i}{L_i^{n}}\bigg).
		\end{multline*}
	\end{theorem}
	The strategy of the proof of Theorem \ref{T:intr} and the outline of the paper can be explained as follows:
	\begin{itemize} 
		\item[(I)]  After introducing some preliminaries and notation used throughout the paper, we define, in \S2.2, the $t$-module $G_n$ given by the tensor product of a Drinfeld $A$-module $\phi$ of rank 2 and the $n$-th tensor power of the Carlitz module. We also discuss effective $t$-motives, Taelman $t$-motives and their $L$-series (see \S2.3, \S 2.4 and \S2.5 for details).
		\item[(II)] In \S3, we analyze the certain entries of the coefficients of the logarithm series $\Log_{G_n}$ of $G_n$. Using Papanikolas' method \cite[Sec. 4.3]{PLogAlg} as well as Lemma \ref{L:1} and Proposition \ref{P:1}, we relate them to shadowed partitions and Carlitz logarithm coefficients (Corollary \ref{C:1}). We also detect some elements living in the convergence domain of the function $\Log_{G_n}$ induced by the logarithm series of $G_n$ (Theorem \ref{T:T1}).
		\item[(III)] In \S4, we introduce the unit module $U(G_n/A)$ of $G_n$ (see Definition \ref{D:unit}) and recall some results on invertible lattices which are due to Debry \cite[Sec. 2]{Deb}. Combining them with Theorem \ref{T:T1}, we give the generators of the unit module $U(G_n/A)$ as an $A$-module in terms of the values of the logarithm function $\Log_{G_n}$ at some algebraic points  (Theorem \ref{T:unit}).
		\item[(IV)] In \S5, we apply the work of Angl\`{e}s,  Ngo Dac and Tavares Ribeiro \cite{ADTR} to our construction. We also study the Taelman $L$-values and show how they are related to the special values of Goss $L$-series (Proposition \ref{P:lseries}).  Finally, we formulate the special value $L(M_\phi,n+1)$ and prove Theorem \ref{T:intr} by using Theorem \ref{T:unit}. 
	\end{itemize}

	\begin{remark} 
		Assume that $\varrho$ is a Drinfeld $A$-module of $r\geq 2$ given by $\varrho_{\theta}=\theta+A_1\tau+\dots+A_r\tau^r.$ Although our arguments introduce a way to generalize Theorem \ref{T:intr} when $\varrho$ is defined with respect to the condition that $A_i\in \mathbb{F}_q$ for all $1\leq i \leq r-1$ and $A_r\in \mathbb{F}_q^{\times}$, our current method does not allow us to prove similar results when some of the coefficients of $\varrho$ is in $A\setminus \mathbb{F}_q$. This is due to the difficulty of understanding the generators of $U(G_n/A)$ in that case. One can also generalize Theorem \ref{T:intr} for larger values of $n$, if a version of Proposition \ref{P:lattice} for such values of $n$ is understood (see \S5.3 for details).  We hope to tackle these problems in the near future.
	\end{remark}
		\subsection*{Acknowledgments} The author is thankful to Chieh-Yu Chang, Yen-Tsung Chen, Tuan Ngo Dac, Nathan Green, Yoshinori Mishiba and Changningphaabi Namoijam for useful suggestions and fruitful discussions. The author also thanks the referee for reading the manuscript carefully and all the suggestions to make the content of the present paper clearer.

	\section{Preliminaries} 
	\subsection{Hyperderivatives}
	For any non-negative integers $i$ and $j$, the binomial coefficient $\binom{i}{j}$ is given by
	\[
	\binom{i}{j}:=\begin{cases} \frac{i!}{(i-j)!j!} & \text{ if } i\geq j\\
	0 & \text{ if } i<j
	\end{cases}.
	\]
	Furthermore, when $k=-i$ is a negative integer, we define
	\[
	\binom{k}{j}=(-1)^j\binom{i+j-1}{j}.
	\]
	We now define the  $j$-th hyperdifferential operator $\partial^j_{\theta}:K_{\infty}\to K_{\infty}$ with respect to $\theta$ by 
	\[
	\partial_{\theta}^{j}\Big(\sum_{k\leq k_0}c_k\theta^k\Big):=\sum_{k\leq k_0}c_k\binom{k}{j}\theta^{k-j}\text{ , }c_k\in \mathbb{F}_q.
	\]
	Note that if $j=0$, then $\partial_{\theta}^{0}(g)=g$ for all $g\in K_{\infty}$. Let  $\CC_{\infty}((t))$ be the field of formal Laurent series in $t$ with coefficients in $\CC_{\infty}$. For any $j\in \mathbb{Z}_{\geq 0}$, we define the $j$-th hyperdifferential operator $\partial^{j}_{t}:\CC_{\infty}((t))\to \CC_{\infty}((t))$ with respect to $t$ by
	\[
	\partial^{j}_{t}\bigg(\sum_{i=i_0}^{\infty}g_it^i\bigg):=\sum_{i=i_0}^{\infty}g_i\binom{i}{j}t^{i-j}\text{, }g_i\in \CC_{\infty}.
	\]
	Furthermore, when $j=0$, we have $\partial_{t}^{0}(f)=f$ for any $f\in \CC_{\infty}((t))$ and if $f_1,f_2\in \CC_{\infty}((t))$ and $n\geq 0$, we have the following product rule:
	\begin{equation}\label{E:rule}
	\partial_{t}^{n}(f_1f_2)=\sum_{\substack{j_1,j_2\geq 0\\j_1+j_2=n}}\partial_{t}^{j_1}(f_1)\partial_{t}^{j_2}(f_2).
	\end{equation}
	For more details on hyperderivatives, we refer the reader to \cite{B99}, \cite{Con} and \cite{F47}. 
	
	The next proposition is useful to deduce our results relating to the hyperderivatives.
	\begin{proposition} \cite[Cor. 2.7]{US98}, \cite[Prop. 3.3.2]{CGM} \label{P:CGM}  Consider the power series
		$
		f=\sum_{i=0}^{\infty}a_i(t-\theta)^{i} \in \CC_{\infty}[[t]]
		$
		so that, as a function of $t$, it is convergent in $D_q:=\{z\in \CC_{\infty} | \ \  \inorm{z}\leq q\}$. Then for any $j\geq 0$, we have
		\[
		a_j=\partial_{t}^{j}(f)_{|t=\theta}.
		\]
	\end{proposition}
	
	\subsection{The $t$-module $G_n$}
	We start with the definition of $t$-modules and then analyze  the tensor product of certain $t$-modules which takes our interest throughout the paper. For further details on $t$-modules and their tensor products, we refer the reader to \cite{BP02}, \cite{H93}, \cite{HartlJuschka16} and  \cite{Juschka10}.
	
	Let $k,m\in \ZZ_{\geq 1}$. For any matrix $B=(B_{i,j})\in \Mat_{k\times m}(L)$ and integer $d$, set $B^{(d)}:=(B_{i,j}^{q^d})$. Furthermore, we extend the norm $\inorm{\cdot}$ to $\Mat_{k\times m}(L)$ by setting $\inorm{B}:=\sup_{i,j}\inorm{B_{i,j}}$. Consider the set $\Mat_{k\times m}(L)[[\tau]]:=\{\sum_{i\geq 0}B_i\tau^i \ \ | B_i\in \Mat_{m\times k}(L)\}$ of twisted power series. When $k=m$, we define the non-commutative ring  $\Mat_{k}(L)[[\tau]]:=\Mat_{k\times k}(L)[[\tau]]$ subject to the condition
	\[
	\tau B=B^{(1)}\tau
	\]
	 for all $B\in \Mat_{k}(L)[[\tau]]$ and set $\Mat_{k}(L)[\tau]\subset\Mat_{k}(L)[[\tau]] $ to be the subring of polynomials in $\tau$.
	\begin{definition} 
		\begin{itemize} 
			\item[(i)] A $t$-module of dimension $d$ is a tuple $G:=(\mathbb{G}^{d}_{a/L},\psi)$ where $\mathbb{G}^{d}_{a/L}$ is the $d$-dimensional additive algebraic group over $L$ and $\psi$ is an $\mathbb{F}_q$-linear ring homomorphism $\psi:A\to \Mat_{d}(L)[\tau]$ defined by 
			\begin{equation}\label{E:red}
			\psi(\theta):=A_0+A_1\tau+\dots+A_m\tau^m
			\end{equation}
			such that $A_0=\theta \Id_{d}+N$ where $\Id_{d}$ is the $d\times d$ identity matrix and $N$ is a nilpotent matrix. For each $0\leq i \leq m$, if $A_i$ is in $\Mat_{d}(R)$ where $R$ is a subring of $L$, then we say $G$ is defined over $R$.
			\item[(ii)] Morphisms between $t$-modules $G=(\mathbb{G}^{d_1}_{a/L},\psi_1)$ and $G^{\prime}=(\mathbb{G}^{d_2}_{a/L},\psi_2)$ are given by any element $\Psi\in \Mat_{d_2\times d_1}(\CC_{\infty})[\tau]$ satisfying
			\[
			\Psi \psi_1(\theta)=\psi_2(\theta)\Psi.
			\]
			We also denote the category of $t$-modules by $\mathcal{G}$.
		\end{itemize}
	\end{definition}
	\begin{example}\label{Ex:tmodules}
		\begin{itemize}
			\item[(i)] Any Drinfeld $A$-module $\phi$ defined as in \eqref{D:drinfeld} can be considered as a $t$-module of dimension one defined over $L$. 
			\item[(ii)] Let $n$ be a positive integer. Another example of $t$-modules can be given by the Carlitz $n$-th tensor power $C^{\otimes n}:=(\mathbb{G}^{n}_{a/K},\psi)$ where $\psi$ is the $\mathbb{F}_q$-linear ring homomorphism given by 
			
			\[
			\psi(\theta)=\begin{bmatrix}
			\theta&1& & \\
			& \ddots&\ddots & \\
			& & \theta & 1\\
			& & & \theta  
			\end{bmatrix}+\begin{bmatrix}
			0&\dots&\dots &0 \\
			\vdots& & &\vdots \\
			0& &  & \vdots\\
			1&0&\dots& 0 
			\end{bmatrix}\tau.
			\] 
			Note that when $n=1$, the definition of $C^{\otimes 1}$ coincides with the Carlitz module. We refer the reader to \cite{AndThak90} for further details. 
		\end{itemize}
	\end{example}
	
	Let $R$ be an $\mathbb{F}_q$-algebra containing each entry of  $A_0,A_1,\dots,A_m$. The $A$-module action  on $\Mat_{d\times 1}(R)$ induced by  $G=(\mathbb{G}^{d}_{a/L},\psi)$ is given as
	\[
	\theta\cdot x:=\psi(\theta)x:=A_0x+A_1x^{(1)}+\dots +A_mx^{(m)}, \ \ x\in\Mat_{d\times 1}(R)
	\]
	and denote such $A$-module by $G(R)$. Furthermore, we set $\partial_{\psi}(\theta):=A_0$ and  define the $A$-module action on $\Mat_{d\times 1}(R)$ via the map $\partial_{\psi}:A\to \Mat_{d}(R)$ so that 
	\begin{equation}\label{E:modact}
	\theta\cdot x:=\partial_{\psi}(\theta)x:=A_0x=(\theta\Id_{d}+N)x,\ \  x\in \Mat_{d\times 1}(R)
	\end{equation}
	and denote such $A$-module by $\Lie(G)(R)$.
	
	When $L=K$ and $R'$ is any subring of $\CC_{\infty}$ containing $K_{\infty}$, by using \cite[Lem. 1.7]{Fang}, the $A$-module action induced by $\partial_{\psi}$ as in \eqref{E:modact} can be uniquely extended to a $K_{\infty}$-vector space action on $\Mat_{d\times 1}(R')$  via the map $\partial_{\psi}:K_{\infty}\to \Mat_{d}(K_{\infty})$ defined by
	\begin{equation}\label{E:exten}
	\partial_{\psi}\Big(\sum_{i\geq i_0}c_i\theta^{-i}\Big)=\sum_{i\geq i_0}c_i(\theta\Id_{d}+N)^{-i} \ ,\ c_i\in \mathbb{F}_q
	\end{equation}
	so that $f\cdot x:=\partial_{\psi}(f)x$ for any $f\in K_{\infty}$ and $x\in \Mat_{d\times 1}(R')$. We denote such $K_{\infty}$-module by $\Lie(G)(R')$.
	
	One can assign an exponential series $\Exp_{G}\in \Mat_{d}(L)[[\tau]]$ to any $t$-module $G$ given by
	\[
	\Exp_{G}:=\sum_{i=0}^{\infty}Q_i\tau^i, \text{ } Q_0=\Id_{d} \text{, } Q_i\in \Mat_{d}(L)
	\]
	subject to the condition that  $\Exp_{G}\partial_{\psi}(\theta)=\psi(\theta)\Exp_{G}$. The exponential series $\Exp_{G}$ induces to an everywhere convergent and vector valued $\mathbb{F}_q$-linear homomorphism $\Exp_{G}:\Lie(G)(\CC_{\infty})\to G(\CC_{\infty})$ defined by 
	\[
	\Exp_{G}(u)=\sum_{i\geq 0}Q_i u^{(i)} \ \ , \ \ u\in \Mat_{d\times 1}(\CC_{\infty}).
	\]

	The logarithm series $\Log_{G}\in \Mat_{d}(L)[[\tau]]$ of $G$ which is the formal inverse of $\Exp_{G}$ is given by
	\[
	\Log_{G}:=\sum_{i=0}^{\infty}P_i\tau^i, \text{ } P_0=\Id_{d} \text{, } P_i\in \Mat_{d}(L)
	\]
	with respect to the condition
	\begin{equation}\label{E:funceq}
	\partial_{\psi}(\theta)\Log_{G}=\Log_{G}\psi(\theta).
	\end{equation}
	Similar to the exponential series, the logarithm series $\Log_{G}$ induces to an $\mathbb{F}_q$-linear homomorphism $\Log_{G}:\mathbb{D}\to \Lie(G)(\CC_{\infty})$ defined by 
	\[
	\Log_{G}(u)=\sum_{i\geq 0}P_i u^{(i)} \ \ , \ \ u\in \mathbb{D}
	\]
	where $\mathbb{D}$ is the domain of convergence of $\Log_{G}$ in $G(\CC_{\infty})$ (see \cite[Lem. 2.5.4]{HartlJuschka16} for more details on $\mathbb{D}$).
	
	We now fix a positive integer $n$ and the Drinfeld $A$-module $\phi$ of rank 2 defined by
	\begin{equation}\label{E:rank2}
	\phi_{\theta}=\theta+a\tau+b\tau^2
	\end{equation}
	for $a\in A$ and $b\in A\setminus\{0\}$.  We introduce the $t$-module $G_n=(\mathbb{G}_{a/K}^{2n+1},\phi_n)$, constructed from $\phi$ and $C^{\otimes n}$, where $\phi_n$ is the $\mathbb{F}_q$-linear ring homomorphism 
	$
	\phi_n:A\to \Mat_{2n+1}(K)[\tau]
	$ given by 
	\begin{equation}\label{E:tmodule}
	\phi_{n}(\theta)=\theta \Id_{2n+1}+N+E\tau
	\end{equation}
	such that $N\in \Mat_{2n+1}(\mathbb{F}_q)$ and $E\in \Mat_{2n+1}(A)$ are defined as
	\[
	N:=\begin{bmatrix}
	0&0&1&\dots&\dots&0\\
	&0&0&1&\dots&0\\
	& &\ddots&\ddots&\ddots& \vdots\\
	& &  & 0&0&1\\
	& &  & & 0&0\\
	& & & & & 0
	\end{bmatrix}\text{, }E:=\begin{bmatrix}
	0&\dots&\dots & \dots &0\\
	\vdots& & & &\vdots \\
	0& & & & \vdots\\
	1&0&\dots&\dots &\vdots\\
	a&b&0&\dots &0
	\end{bmatrix}.
	\]
	
	For any $f\in K_{\infty}$, we also set $d_n[f]\in \Mat_{2n+1}(K_{\infty})$ given by 
	\[
	d_n[f]:=\begin{bmatrix}
	f&0&\partial_{\theta}(f)&0&\partial^2_{\theta}(f)&\dots &0&\partial_{\theta}^{n}(f)\\
	&f&0&\partial_{\theta}(f)& & \\
	& & \ddots &\ddots  &\ddots  & \\
	& & & \ddots &\ddots  &\ddots&  & \\
	& & & &\ddots &\ddots  &\ddots  & \\
	
	& &  & & & f&0&\partial_{\theta}(f)\\
	& &  & & & & f&0\\
	& & & & & & &f
	\end{bmatrix}.
	\]
	\begin{remark} It is important to emphasize that our definition for the matrix $d_{n}[f]$ is slightly different than the $d$-matrices of Papanikolas defined in \cite[Eq. (2.5.1)]{PLogAlg}.
	\end{remark}
	\begin{lemma}\label{L:dpart} For any $a\in A$, we have $\partial_{\phi_{n}}(a)=d_n[a]$.
	\end{lemma}
	\begin{proof} By the $\mathbb{F}_q$-linearity of the action $\partial_{\phi_{n}}$ on $A$ and hyperdifferential operators with respect to $\theta$, it is enough to prove the lemma for $a=\theta^i$ for $i\in \mathbb{Z}_{\geq 0}$. We do induction on $i$. If $i=0$, then we are done. Assume that the assumption holds for all $i$. Note that for positive integers $i$ and $j$ such that $i\geq j$ we have
		\begin{equation}\label{E:eqbinom}
		\binom{i}{j}+\binom{i}{j-1}=\binom{i+1}{j}.
		\end{equation} 
		Using \eqref{E:eqbinom}, we obtain
		\begin{equation}\label{E:hyp}
		\theta\partial^j_{\theta}(\theta^i)+\partial_{\theta}^{j-1}(\theta^i)=\partial_{\theta}^{j}(\theta^{i+1}).
		\end{equation}
		Using the fact that $\phi_n$ is an $\mathbb{F}_q$-linear ring homomorphism, we obtain $\partial_{\phi_{n}}(\theta^{i+1})=\partial_{\phi_{n}}(\theta)\partial_{\phi_{n}}(\theta^i)$. Thus the equality in  \eqref{E:hyp} implies the assumption for $i+1$ as desired.
	\end{proof}

	\subsection{Effective $t$-motives over $K$}  We define $K[t]$ to be the commutative polynomial ring consisting of polynomials in $t$ with coefficients in $K$ and $K(t)$ to be its quotient field. 
	For any $f=\sum_{i\geq 0} c_it^i\in K[t]$ and $j\in \mathbb{Z}$, we set $f^{(j)}:=\sum_{i\geq 0} c_i^{q^j}t^i$.
	We define the non-commutative ring $K[t,\tau]:=K[t][\tau]$ subject to the condition
	\[
	\tau f=f^{(1)}\tau,  \ \  f\in K[t].
	\]
	
	\begin{definition} 
		\begin{itemize}
			\item[(i)] An effective $t$-motive $M$ defined over $K$ is a left $K[t,\tau]$-module which is free and finitely generated over $K[t]$ such that the determinant of the matrix representing the $\tau$-action on $M$ with respect to any chosen $K[t]$-basis is equal to $c(t-\theta)^{s}$ for some $c\in K^{\times}$ and $s\in \mathbb{Z}_{\geq 0}$.
			\item[(ii)] The morphisms between effective $t$-motives are left $K[t,\tau]$-module homomorphisms and we let $\mathbf{M}$ be the category of effective $t$-motives defined over $K$. For any $M_1,M_2\in \mathbf{M}$, we denote the set of morphisms between $M_1$ and $M_2$ by $\Hom_{\mathbf{M}}(M_1,M_2)$.
		\end{itemize}
	\end{definition}

	Now for a given effective $t$-motive $M$ which is also free and finitely generated over $K[\tau]$, let $\{v_1,\dots,v_d\}$ be a fixed $K[\tau]$-basis of $M$. Then there exists a matrix $\Phi_{\theta}\in \Mat_{d}(K)[\tau]$ such that 
	\[
	t \begin{pmatrix}
	v_1\\
	\vdots\\
	v_d
	\end{pmatrix}=\Phi_{\theta}\begin{pmatrix}
	v_1\\
	\vdots\\
	v_d
	\end{pmatrix}.
	\]
	Thus we can define an $\mathbb{F}_q$-linear ring homomorphism $\Phi:A\to \Mat_{d}(K)[\tau]$ by $\Phi(\theta):=\Phi_{\theta}$ so that $(\mathbb{G}^d_{a/K},\Phi)$ forms a $t$-module of dimension $d$. We call the $t$-module $(\mathbb{G}_{a/K}^{d},\Phi)$ formed via this process an abelian $t$-module corresponding to $M$.

	By \cite[Thm. 1]{Taelman3} (see also \cite[Thm. 10.8]{Sta07}), we know that there  is an anti-equivalence of categories between the subcategory  of effective $t$-motives over $K$ which are also finitely generated over $K[\tau]$ and the category of abelian $t$-modules defined over $K$. We now see some examples of such correspondence between effective $t$-motives and abelian $t$-modules. 
	
	\begin{example}\label{Ex1}
		\begin{itemize}
			\item[(i)] Let $\phi$ be the Drinfeld $A$-module of rank 2 given as in \eqref{E:rank2}. We set a left $K[t,\tau]$-module $M_\phi:=K[t]m_1\oplus K[t]m_2$ with some chosen $K[t]$-basis $\{m_1,m_2\}$ of $M_\phi$  whose $\tau$-action is given by 
			\[
			\tau \cdot (f_1
			m_1+f_2m_2) =f_2^{(1)}(t-\theta)b^{-1}m_1+(f_1^{(1)}-f_2^{(1)}ab^{-1})m_2
			\]
			for any $f_1,f_2\in K[t]$. It is the effective $t$-motive corresponding to  $\phi$. One can also easily see that $\{m_1\}$ is a $K[\tau]$-basis for $M_\phi$. 
			\item[(ii)]Let $M$ be an effective $t$-motive which is free of rank $r$ over $K[t]$. The $r$-th exterior power of $M$ is called the determinant of $M$ and is denoted by $\det(M)$. One can easily prove that $\det(M)$ is an effective $t$-motive of rank 1 over $K[t]$. For instance, let $\phi$ be the Drinfeld $A$-module of rank 2 defined as in \eqref{E:rank2}. Then $\det(M_\phi)=K[t]m_1\wedge m_2$ so that $\tau$ acts on $m_1\wedge m_2$ by
			\[
			\tau\cdot f(m_1\wedge m_2)=-f^{(1)}b^{-1}(t-\theta)(m_1\wedge m_2) \ , \ f\in K[t].
			\]
			Observe that $\det(M_\phi)$ is also a free $K[\tau]$-module with the $K[\tau]$-basis $\{m_1\wedge m_2\}$.
			\item[(iii)] Let $n$ be a non-negative integer. We now define the left   $K[t,\tau]$-module  $
			\textbf{C}^{\otimes n}:=K[t]m
			$
			whose $\tau$-action is given by $\tau \cdot (fm)=f^{(1)}(t-\theta)^{n}m$ for any $f\in K[t]$. It is free of rank one with the $K[t]$-basis $\{m\}$ and free of rank $n$ over $K[\tau]$ with the basis $\{m,(t-\theta)m,\dots,(t-\theta)^{n-1}m\}$. One can see that the abelian $t$-module corresponding to $\textbf{C}^{\otimes n}$ is given by $C^{\otimes n}$ defined in Example \ref{Ex:tmodules}(ii). When $n=1$, we also set $\textbf{C}:=\textbf{C}^{\otimes 1}$.
		\end{itemize}
	\end{example}

	Consider the left $K[t,\tau]$-module $M_n:=M_\phi\otimes_{K[t]} \textbf{C}^{\otimes n}$ on which $\tau$ acts diagonally. Using Example \ref{Ex1}, we see that $M_n$ is an effective $t$-motive which is free of rank $2$ over $K[t]$ with the basis $\{v_{1,0},v_{2,0}\}$ introduced as $v_{1,0}:=m_1 \otimes m$ and $v_{2,0}:=m_2 \otimes m$. For any $1\leq j\leq n$, we further let $v_{1,j}:=m_1 \otimes (t-\theta)^{j}m$ and similarly, set $v_{2,j}:=m_2 \otimes(t-\theta)^{j}m$ for $1\leq j \leq n-1$.  Thus using the $K[t]$-basis $\{v_{1,0},v_{2,0}\}$ of $M_n$, we see that the set $\{v_{1,0},\dots,v_{1,n},v_{2,0},\dots,v_{2,n-1}\}$ is a $K[\tau]$-basis for $M_n$ and hence $M_n$ is free of finite rank over $K[\tau]$. Moreover, the following identities hold:
	
	\begin{align*}
	&(t-\theta) v_{1,j}=v_{1,j+1} \text{ for } 0\leq j \leq n-1, \\
	&(t-\theta) v_{2,j}=v_{2,j+1} \text{ for } 0\leq j \leq n-2,\\
	&(t-\theta) v_{2,n-1}=\tau v_{1,0}, \\
	&(t-\theta) v_{1,n}=a\tau v_{1,0}+b\tau v_{2,0}. 
	\end{align*}
	Thus we see that the multiplication by $t$ on $M_n$ is given by 
	\[
	t \begin{pmatrix}
	v_{1,0}\\
	v_{2,0}\\
	\vdots\\
	v_{1,n-1}\\
	v_{2,n-1}\\
	v_{1,n}
	\end{pmatrix}=\phi_{n}(\theta)\begin{pmatrix}
	v_{1,0}\\
	v_{2,0}\\
	\vdots\\
	v_{1,n-1}\\
	v_{2,n-1}\\
	v_{1,n}
	\end{pmatrix}
	\]
	which shows that $G_n=(\mathbb{G}_{a/K}^{2n+1},\phi_n)$ is the $t$-module corresponding to $M_n$. Hence $G_n$ is an abelian $t$-module.

	Let $w$ be a monic irreducible polynomial in $\bA$ and $\bA_{w}$ be the completion of $\bA$ at $w$. Let $M$ be an effective $t$-motive over $K$ and set
	\[
	M_{K^{\text{sep}}}:=M \otimes_{K} K^{\text{sep}}
	\]
	which is a left $K^{\text{sep}}[t,\tau]$-module and the $\tau$-action on $M_{K^{\text{sep}}}$ is given by $\tau(f\otimes g)=\tau(f)\otimes g^q$ for all $f\in M$ and $g\in K^{\text{sep}}$. For any $\mathbb{F}_q[t,\tau]$-module $I$, let $I^{\tau=1}$ be the set of elements of $I$ fixed by the action of $\tau$.  We define the $\bA_w$-module
	\[
	T_{w}(M):=\lim_{\substack{\leftarrow\\i}}(M_{K^{\text{sep}}}/w^iM_{K^{\text{sep}}})^{\tau=1}.
	\]
	Furthermore we set 
	\[
	V_{w}(M):=T_{w}(M)\otimes_{\bA_w}\textbf{K}_{w}
	\]
	which is a finite dimensional $\textbf{K}_w$-vector space with a continuous action of $\Gal(K^{\text{sep}}/K)$ (see \cite[Prop. 1]{Taelman3}). Let $\rho=(\rho_w)$ be the family of homomorphisms $\rho_w:\Gal(K^{\text{sep}}/K)\to \GL(V_{w}(M))$  induced by the action of $\Gal(K^{\text{sep}}/K)$ on  $V_{w}(M)$. The next theorem is due to Gardeyn (see also \cite[Prop. 2]{Taelman3}).
	\begin{theorem}\cite[Thm. 3.3]{G01} \label{T:Gal} The following statements hold.
		\begin{itemize}
			\item[(i)] We have $\dim_{\textbf{K}_w}V_w(M)=\rank_{K[t]}M$.
			\item[(ii)] The family $\rho=(\rho_w)$  forms a strictly compatible system.
		\end{itemize}
	\end{theorem}
	Throughout the present paper, we call $\rho=(\rho_w)$ in Theorem \ref{T:Gal} the family of representations attached to the effective $t$-motive $M$.
	
	\subsection{Taelman $t$-motives}
	We review the properties of a certain category $\mathcal{T}$, which is a rigid $\bA$-linear pre-abelian tensor category \cite[Thm. 2.3.7]{Taelman2} (see also \cite[Sec. 2.2.5]{Tae07}), consisting of Taelman $t$-motives introduced in \cite{Taelman2}.
	
	Let $M_1$ and $M_2$ be effective $t$-motives defined over $K$. The tensor product $M_1\otimes M_2:=M_1\otimes_{K[t]} M_2$ is also an effective $t$-motive on which $\tau$ acts diagonally.
	
	We define  
	$\Hom(M_1,M_2):=\Hom_{K[t]}(M_1 ,M_2)$. Taelman \cite[Prop. 2.2.3]{Taelman2} showed that for sufficiently large $n$, $\Hom(M_1,M_2\otimes \textbf{C}^{\otimes n})$ induces the structure of an effective $t$-motive whose $K[\tau]$-module structure can be described in what follows. 
	
	For $i=1,2$, let $B_{i}\in \Mat_{s_i}(K[t])$ be defined so that
	$\tau\cdot m_i^{\tr}=B_im_i^{\tr}$ where $m_i:=[m_{i,1},\dots,m_{i,s_i}]$
	consists of a $K[t]$-basis elements $m_{i,1},\dots,m_{i,s_i}$ of $M_i$. Let $\overline{K}$ be the algebraic closure of $K$ in $\CC_{\infty}$. We consider the  dual $K[t]$-basis $\{f_{i,j}\in \Hom(M_1,M_2\otimes \textbf{C}^{\otimes n})\ \ | i\in\{1,\dots,s_1\}, j\in \{1,\dots,s_2\}\}$ of $\Hom(M_1 ,M_2\otimes \textbf{C}^{\otimes n})$  given by 
	\[
	f_{i,j}(m_{1,k}):=\begin{cases}m_{2,j}\otimes 1 &\text{ if } k=i\\
	0 & \text{ otherwise}
	\end{cases}.
	\] 
	Note that after the extension of scalars, we have
	\[
	\Hom(M_1,M_2\otimes \textbf{C}^{\otimes n})\subset \Hom_{\overline{K}(t)}(M_1\otimes \overline{K}(t) ,M_2\otimes \textbf{C}^{\otimes n}\otimes\overline{K}(t)).
	\]
	Let $S_{i,j}\in \Mat_{s_2\times s_1}(K[t])$ be the representation matrix of $f_{i,j}$ with respect to $m_{1}$ and $m_2$. Then we define $\tau\cdot f_{i,j}:M_1\to M_2\otimes \textbf{C}^{\otimes n}$ to be the $\overline{K}(t)$-module homomorphism  whose representation matrix with respect to $m_1$ and $m_2^{\prime}:=\{m_{2,1}\otimes 1,\dots,m_{2,s_2}\otimes 1\}$  is given by $(t-\theta)^nB_{2}S_{i,j}(B_{1}^{-1})^{\tr}$. Since the determinant of $B_{1}$ is some power of $(t-\theta)$ times a unit in $K$, for sufficiently large $n$, the matrix $(t-\theta)^nB_{2}S_{i,j}(B_{1}^{-1})^{\tr}$ will have coefficients in $K[t]$. Hence $\tau\cdot f_{i,j}$ is indeed a $K[t]$-module homomorphism  in $\Hom(M_1,M_2\otimes \textbf{C}^{\otimes n})$.
	
	We are now ready to give the definition of Taelman $t$-motives. 
	
	\begin{definition}\label{D:D1}
		\begin{itemize}
			\item[(i)] A Taelman $t$-motive $\mathbb{M}$ is a tuple $(M,n)$ where $M$ is an effective $t$-motive defined over $K$ and $n\in \mathbb{Z}$.
			\item[(ii)] We define the set of morphisms between Taelman $t$-motives $(M_1, n_1)$ and $(M_2,n_2)$ by
			\[
			\Hom_{\mathcal{T}}((M_1, n_1),(M_2 ,n_2)):=\Hom_{\textbf{M}}(M_1\otimes  \textbf{C}^{\otimes (n+n_1)},M_2\otimes  \textbf{C}^{\otimes (n+n_2)})
			\]
			where $n\geq \max\{-n_1,-n_2\}$.
			\item[(iii)] For any $c\in K^{\times}$, we define $c\textbf{1}:=(c\textbf{1},0)$ to be the Taelman $t$-motive where $c\textbf{1}=K[t]$ on which $\tau$ acts as $\tau\cdot f=cf^{(1)}$ for any $f\in \textbf{1}$. When $c=1$, we call $\textbf{1}$ the trivial Taelman $t$-motive.
		\end{itemize}
	\end{definition}
	\begin{remark}\label{R:rem1}
		\begin{itemize}
			\item [(i)] It is important to point out that for effective $t$-motives $M_1$ and $M_2$, the canonical isomorphism
			\[
			\Hom_{\textbf{M}}(M_1,M_2)\cong\Hom_{\textbf{M}}(M_1\otimes \textbf{C},M_2\otimes \textbf{C})
			\]
			actually shows that the definition of morphisms between the objects of $\mathcal{T}$ is independent of $n$. 
			\item[(ii)] The category $\textbf{M}$ of effective $t$-motives can be embedded into $\mathcal{T}$ as a subcategory via the fully faithful functor $M \to (M,0)$ and by the abuse of notation, we continue to denote the image of $M$ under this functor by the same notation.
		\end{itemize}
	\end{remark}
	For any Taelman $t$-motive $\mathbb{M}_1:=(M_1,i_1)$ and $\mathbb{M}_2:=(M_2,i_2)$, we define
	\begin{equation}\label{E:tensdef}
	\mathbb{M}_1\otimes \mathbb{M}_2:=(M_1\otimes M_2, i_1+i_2).
	\end{equation}
	Note that $
	\mathbb{M}_1\otimes \mathbb{M}_2=\mathbb{M}_2\otimes \mathbb{M}_1 
	$ and moreover, for $\mathbb{M}\in \mathcal{T}$, we obtain $\mathbb{M}\otimes\textbf{1}=\textbf{1}\otimes \mathbb{M}=\mathbb{M}$.
	
	We define the internal hom in $\mathcal{T}$ by 
	\[
	\Hom(\mathbb{M}_1,\mathbb{M}_2)=\Hom((M_1,i_1),(M_2,i_2)):=(\Hom(M_1, M_2\otimes \textbf{C}^{\otimes i_2-i_1+i}),-i)
	\]
	where $i\in \mathbb{Z}_{\geq 0}$ is sufficiently large. For an effective $t$-motive $M$, we have the natural isomorphism between 
	$M\otimes\textbf{C}^{j}\otimes\textbf{C}$ and $M\otimes\textbf{C}^{j+1}$ for any $j\geq 0$ which implies that 
	\begin{equation}\label{E:isomorphy}
	(M\otimes \textbf{C},i)\cong (M,i+1), \ \ i\in \mathbb{Z}
	\end{equation}
	by Definition \ref{D:D1}(ii). Moreover for sufficiently large $i$ and $M_1,M_2\in \textbf{M}$, we have 
	\[
	\Hom(M_1,M_2\otimes \textbf{C}^{i})\otimes \textbf{C}\cong \Hom(M_1,M_2\otimes \textbf{C}^{i+1}).
	\]
	Thus one can show that the definition of the internal hom above is actually independent of $i$ and well-defined up to isomorphism of Taelman $t$-motives.

	For Taelman $t$-motives $\mathbb{M}_1,\dots,\mathbb{M}_4$, we have 
	\begin{equation}\label{E:can2}
	\Hom(\mathbb{M}_1,\mathbb{M}_3)\otimes \Hom(\mathbb{M}_2,\mathbb{M}_4)\cong \Hom(\mathbb{M}_1\otimes \mathbb{M}_2,\mathbb{M}_3\otimes \mathbb{M}_4).
	\end{equation}

	Furthermore, we define the dual $\mathbb{M}^{\vee}$ of the Taelman $t$-motive $\mathbb{M}$  by 
	\[
	\mathbb{M}^{\vee}:=\Hom(\mathbb{M},\textbf{1}).
	\]
	
	Taking dual of Taelman $t$-motives is also reflexive in the sense that $(\mathbb{M}^{\vee})^{\vee}=\mathbb{M}$ for any $\mathbb{M}\in \mathcal{T}$. 
	
	Remark \ref{R:rem1}(ii) explains how to identify an effective $t$-motive inside the category $\mathcal{T}$. Now we briefly discuss such identification for $\Hom(M_1,M_2)$ when $M_1,M_2\in \textbf{M}$ up to isomorphism of Taelman $t$-motives: We already know that for sufficiently large $n$, $M':=\Hom(M_1,M_2\otimes \textbf{C}^{\otimes n})$ is an effective $t$-motive. Thus, by the definition of internal hom, we see that $\Hom(M_1,M_2)$ can be identified by the tuple $(M',-n)$ inside $\mathcal{T}$. Some examples are in order.
	\begin{example} \label{Ex:1}
		\begin{itemize}
			\item[(i)] 
			For any positive integer $n$, consider the effective $t$-motive $\textbf{C}^{\otimes n}$. One can easily show that
			\[
			\Hom(\textbf{C}^{\otimes n},\textbf{C}^{\otimes n})\cong\textbf{1}.
			\]
			In other words, $(\textbf{C}^{\otimes n})^{\vee}=\Hom(\textbf{C}^{\otimes n},\textbf{1})$ can be  identified by $(\textbf{1},-n)$ in $\mathcal{T}$. We also note that by \eqref{E:isomorphy}, $\textbf{C}^{\otimes n}$ can be also identified by $(\textbf{1},n)\cong (\textbf{C}^{\otimes (n-i)},i)$ for any  $i\in \{0,\dots, n-1\}$. Furthermore, using \eqref{E:tensdef}, one can see that 
			\begin{equation}\label{E:dualC}
			(\textbf{C}^{\otimes n})^{\vee}\otimes \textbf{C}^{\otimes m}=(\textbf{1},m-n) \ , \ m\in \mathbb{Z}_{\geq 0}.
			\end{equation}
			\item[(ii)] Let $\tilde{\phi}$ be the Drinfeld $A$-module given by $\tilde{\phi}_{\theta}=\theta -ab^{-1}\tau+b^{-1}\tau^2$ such that $a\in \mathbb{F}_q$ and $b\in \mathbb{F}_q^{\times}$. Using the $K[\tau]$-module structure on $\Hom(M_{\tilde{\phi}},\textbf{C})$, one can see that 
			$
			\Hom(M_{\tilde{\phi}},\textbf{C})\cong M_{\phi}
			$
			where $\phi$ is the Drinfeld $A$-module as in \eqref{E:rank2} with $a\in \mathbb{F}_q$ and $b\in \mathbb{F}_q^{\times}$.  Thus, $M_{\tilde{\phi}}^{\vee}$ is identified by $(M_{\phi},-1)=M_{\phi}\otimes \textbf{C}^{\vee}$ where the equality follows from the previous example and \eqref{E:tensdef}. Moreover, since $M_{\tilde{\phi}}\otimes (-b^{-1}\textbf{1})\cong M_{\phi}$, using \eqref{E:tensdef}, one further sees that  
			\begin{equation}\label{E:prten11}
			M_{\tilde{\phi}}^{\vee}=(M_{\phi},-1)=(M_{\tilde{\phi}},0)\otimes(-b^{-1}\textbf{1},-1)=M_{\tilde{\phi}}\otimes \det(M_{\tilde{\phi}})^{\vee}.
			\end{equation}
		\end{itemize}
	\end{example}
	
	\subsection{$L$-series of Taelman $t$-motives and Taelman $L$-values} Using Theorem \ref{T:Gal} and the tensor compatibility of the functor $V_w$  defined in \S2.3 for any monic irreducible element $w\in \bA$, 
	in \cite[Sec. 2.8]{Taelman3}, Taelman was able to introduce the $L$-function $L(\mathbb{M},\cdot)$ corresponding to a Taelman $t$-motive $\mathbb{M}$ satisfying the property 
	\begin{equation}\label{E:Lser}
	L(\mathbb{M}\otimes \textbf{C},s+1)=L(\mathbb{M},s), \ \ s\in \mathbb{Z}
	\end{equation}
	provided that both sides of the identity converge. 
	
	As our first example, for any $s\in \mathbb{Z}$, we define the $L$-series $L(\textbf{1},s)$ corresponding to the trivial Taelman $t$-motive $\textbf{1}$ by
	\[
	L(\textbf{1},s):=\prod_{v\in A_{+}}(1-v^{-s})^{-1}=\sum_{a\in A_{+}}\frac{1}{a^s}\in K_{\infty}
	\] 
	where the product runs over irreducible elements in $A_{+}$ and it converges for any positive integer $s$ (see \cite[Sec. 8]{Goss}).
	
	In this subsection and the rest of the paper, for any positive integer $n$, we are mainly interested in the $L$-function $L(M_n,\cdot)$ of the effective $t$-motive $M_n$ defined in \S2.3 corresponding to  $G_n$ given in \eqref{E:tmodule}. Let $\rho=(\rho_w)$ be the family of homomorphisms $\rho_w:\Gal(K^{\text{sep}}/K)\to \GL_2(V_{w}(M_n))$ induced by the action of $\Gal(K^{\text{sep}}/K)$ on  $V_{w}(M_n)$. We know by Theorem \ref{T:Gal} that $\rho$ indeed forms a strictly compatible system and hence one can define the $L$-function $L(M_n,\cdot):=L((M_n,0),\cdot)=L_{U'}(\rho,\cdot)$ as in \eqref{E:lfunct}. We recall that the values of our $L$-function converge in $K_{\infty}$ simply after replacing the variable $t$ with $\theta$ as explained in \S1.2.  Furthermore one can check that the exceptional set $U'$ of primes of $A_{+}$ in this case is empty. 
	
	Since $M_{n}\cong (M^{\prime},m)$ for some effective $t$-motive $M^{\prime}$ and $m\in \mathbb{Z}$, by using \eqref{E:Lser}, one can recover values of $L(M_{n}^{\vee},s)$ in terms of $L(M^{\prime},s)$ whenever they are convergent. Indeed by \cite[Prop. 8]{Taelman}, we know that $L(M_n^{\vee},s)$ converges to an element in $K_{\infty}$ for any integer $s\geq 0$. 
	
	Before we finish this subsection, we introduce the Taelman $L$-value corresponding to an abelian $t$-module $G=(\mathbb{G}_{a/K}^{d},\psi)$ which plays a fundamental role to prove our main result. We refer the reader to \cite{Fang} and \cite{Taelman} for further details.
	
	For any finite $A$-module $M$, we set 
	\[
	|M|_{A}:=\det_{\mathbb{F}_q[X]}((1\otimes X)\Id - (\theta\otimes 1) |\ \ M\otimes_{\mathbb{F}_q}\mathbb{F}_q[X])_{|X=\theta}
	\]
	which is the characteristic polynomial of the map $\theta\otimes1$ on $M$ evaluated at $X=\theta$. 
	
	Let $B=(b_{i,j})\in \Mat_d(A)$ and $v\in A_{+}$ be a prime. We define the matrix $\overline{B}:=(\overline{b}_{i,j})\in \Mat_d(A/vA)$ where $\overline{b}_{i,j}\equiv b_{i,j} \pmod{v}$. For any $1\leq j \leq m$ and $x=[x_1,\dots,x_d]^{\tr}\in (A/vA)^d$, we set $x^{(j)}:=[x_1^{q^j},\dots,x_d^{q^j}]^{\tr}$. 
	
	We define $\Lie(G)(A/vA)$ to be the direct sum $(A/vA)^d$ of $d$-copies of $A/vA$ equipped with the $A$-module action given by
	\[
	\theta\cdot x:=\overline{\partial_{\psi}(\theta)}x, \ \ x\in (A/vA)^d.
	\] 
	Similarly, we define $G(A/wA)$ as $(A/wA)^d$ with the $A$-module action given by 
	\[
	\theta\cdot x:=\overline{A}_0x+\dots+\overline{A}_mx^{(m)}, \ \ x\in (A/vA)^d.
	\]

	Now following \cite{Fang}, we define the Taelman $L$-value $L(G/A)$ by the  infinite product
	\[
	L(G/A):=\prod_{v}\frac{|\Lie(G)(A/wA)|_{A}}{|G(A/wA)|_{A}}\in 1+\frac{1}{\theta}\mathbb{F}_q\Big[\Big[\frac{1}{\theta}\Big]\Big]
	\]
	where $v$ runs over all irreducible elements in $A_{+}$.

	\section{The Analysis on the Logarithm series $\Log_{G_n}$} 
	We fix a positive integer $n$ and a Drinfeld $A$-module $\phi$ of rank 2 given by $\phi_{\theta}=\theta+a\tau +b\tau^2$ where $a\in A$ and $b\in A\setminus\{0\}$ unless otherwise stated. We recall the definition of the $t$-module $G_n=(\mathbb{G}^{2n+1}_{a/K},\phi_n)$ from \eqref{E:tmodule}  and denote its logarithm series $\Log_{G_n}$ by 
	\[
	\Log_{G_n}=\sum_{i=0}^{\infty}P_i\tau^i, \text{ } P_0=\Id_{2n+1} \text{, } P_i=(P_{i,(j,k)})\in \Mat_{2n+1}(K).
	\]
	In this section, we analyze the coefficients $P_i$ and using Papanikolas' method in \cite[Sec. 4.1 and 4.3]{PLogAlg}, we determine certain elements lying in the convergence domain of the function $\Log_{G_n}$ induced by the logarithm series of $G_n$ if the coefficients of $\phi$ have certain conditions.
	\begin{lemma}[{cf. \cite[Lem. 4.1.1]{PLogAlg}}]\label{L:1} For any $i\geq 1$ and $j=0,1$, we consider $r_{1,j,i-1}:=P_{i-1,(2n+j,2n)}+a^{q^{i-1}}P_{i-1,(2n+j,2n+1)}\in K$ and define  $r_{2,j,i-1}:=b^{q^{i-1}}P_{i-1,(2n+j,2n+1)}\in K$. Let $R_{i,j} \in \Mat_{1\times 2n+1}(K)$ be given as
		\begin{multline*}
		R_{i,j}:=\Big[\frac{(-1)}{[i]}r_{1,j,i-1},\frac{(-1)}{[i]}r_{2,j,i-1},\dots,\\
		\frac{(-1)^n}{[i]^n}r_{1,j,i-1},\frac{(-1)^n}{[i]^n}r_{2,j,i-1},\frac{(-1)^{n+1}}{[i]^{n+1}}r_{1,j,i-1}\Big].
		\end{multline*}
		Then the $(2n)$-th and $(2n+1)$-st row of $P_i$ are given by $R_{i,0}$ and $R_{i,1}$ respectively.
	\end{lemma}
	\begin{proof}
		Recall that $\phi_n(\theta)=\theta \Id_{2n+1}+N+E\tau$.
		For any two matrices $B_1,B_2\in \Mat_{2n+1}(K)$, set $[B_1,B_2]:=B_1B_2-B_2B_1$. Then we define $\ad(B_1)^{0}(B_2):=B_2$ and for $j\geq 1$, $\ad(B_1)^{j}(B_2)=[B_1,\ad^{j}(B_1)^{j-1}(B_2)]$. Using \cite[Lem. 3.4]{G19} and a similar argument as in \cite[Eq. 3.2.4]{CMApril17}, we have
		\begin{align*}
		P_i&=-\sum_{j=0}^{2(n+1)-2}\frac{\ad(N)^j(P_{i-1}E^{(i-1)})}{[i]^{j+1}}\\
		&=-\sum_{j=0}^{2(n+1)-2}\sum_{m=0}^j(-1)^{j-m}\binom{j}{m}\frac{N^mP_{i-1}E^{(i-1)}N^{j-m}}{[i]^{j+1}}.
		\end{align*}
		Note that $N^m=0$ when $m\geq n+1$ and $N^{j-m}=0$ if $j-m\geq n+1$. Therefore we see that 
		\begin{equation}\label{E:eq111}
		\begin{split}
		P_i&=-\sum_{m=0}^{n}\sum_{j=m}^{n+m}(-1)^{j-m}\binom{j}{m}\frac{N^mP_{i-1}E^{(i-1)}N^{j-m}}{[i]^{j+1}}\\
		&=\sum_{l=1}^{n+1}\sum_{m=0}^n(-1)^l\binom{l+m-1}{m}\frac{N^mP_{i-1}E^{(i-1)}N^{l-1}}{[i]^{l+m}},
		\end{split}
		\end{equation}
		where the last equality follows from setting $l=j-m+1$.  Since the last two rows of $N$ contain only zeros, one can notice from the direct calculation that the multiplication $N^mP_{i-1}E^{(i-1)}N^{l-1}$ has no contribution to the last two rows of $P_i$ if $m\geq 1$. Thus we only consider the case when $m=0$. Observe that 
		\[
		P_{i-1}E^{(i-1)}N^{l-1}=\begin{bmatrix}
		*&\dots &*&*&*&*&\dots&*\\
		\vdots& &\vdots&\vdots &\vdots&\vdots& &\vdots\\
		*& \dots&* &*&*&*&\dots &*\\
		0&\dots &0&r_{1,0,i-1}&r_{2,0,i-1}&0&\dots &0\\
		0&\dots&0&r_{1,1,i-1}&r_{2,1,i-1}&0&\dots &0
		\end{bmatrix}
		\]
		where the only non-zero elements occur in the $2(l-1)+1$-st and $2(l-1)+2$-nd coordinates of the last two rows. We also mention that when $l=n+1$, the non-zero terms appear only in the last coordinate of the last two rows which are actually the terms corresponding to $2(l-1)+1$-st coordinate when $n=l$. Thus, applying \eqref{E:eq111} together with above observation finishes the proof.
	\end{proof}	
	Let 
	\[
	\log_{\phi}=\sum_{i\geq 0}\gamma_i\tau^i
	\]
	be the logarithm series of $\phi$ defined so that $\gamma_0=1$ and 
	\begin{equation}\label{E:Eq1}
	\theta \log_{\phi}=\log_{\phi}\phi_{\theta}.
	\end{equation}
	
	Recall that the logarithm series $\log_{C}$ of the Carlitz module is defined by $\log_{C}=\sum_{i\geq 0}L_i^{-1}\tau^i$ where $L_0=1$ and $L_i=(-1)^{i}[i][i-1]\dots[1]$. 
	
	We also recall the definition of shadowed partitions and elements $F_i$ for all $i\geq 0$ from \S1.3 and prove the following proposition.
	\begin{proposition}\label{P:1} For any $i\geq 2$, we have 
		\[
		\frac{-1}{[i]}(b^{q^{i-2}}F_{i-2}+a^{q^{i-1}}F_{i-1})=F_i.
		\]
	\end{proposition}
	\begin{proof} Note that for any $\mathcal{U}_1=(S_{1,1},S_{1,2})\in \mathcal{P}_2^{1}(i-2)$ and $\mathcal{U}_2=(S_{2,1},S_{2,2})\in \mathcal{P}_2^{1}(i-1)$ with corresponding components $\mathcal{C}_{\mathcal{U}_1}$ and $\mathcal{C}_{\mathcal{U}_2}$ respectively, the elements $-b^{q^{i-2}}\mathcal{C}_{\mathcal{U}_1}/[i]$ and $-a^{q^{i-1}}\mathcal{C}_{\mathcal{U}_2}/[i]$ are the corresponding components  to  the shadowed partitions $(S_{1,1}, S_{1,2}\cup \{i-2\})$ and $(S_{2,1}\cup \{i-1\},S_{2,2})$ in $\mathcal{P}_2(i)$ respectively and they are actually distinct elements of  $\mathcal{P}_2^{1}(i)$ by definition. Define the map $\alpha: \mathcal{P}_2^{1}(i-2)\sqcup \mathcal{P}_2^{1}(i-1)\to \mathcal{P}_2^{1}(i)$ by
		\[
		\alpha(\mathcal{U}):=\begin{cases}
		(S_1, S_2\cup \{i-2\}) \text{ if } \mathcal{U}=(S_1, S_2)\in P_2^{1}(i-2)\\
		(S_1\cup \{i-1\}, S_2) \text { if } \mathcal{U}=(S_1, S_2)\in P_2^{1}(i-1).
		\end{cases}
		\]
		By the above discussion $\alpha$ is  injective. Furthermore, it is also surjective as any element $\mathcal{U}=(S^{\prime}_1,S^{\prime}_2)$ of $\mathcal{P}_2^{1}(i)$  has the property that either $\{i-1\}\in S_1^{\prime}$ or $\{i-2\}\in S_2^{\prime}$. For the former case we have  $\alpha((S^{\prime}_1\setminus\{i-1\},S_2^{\prime}))=\mathcal{U}$ and for the latter case $\alpha((S^{\prime}_1,S_2^{\prime}\setminus\{i-2\}))=\mathcal{U}$. Thus  the proof is completed after summing the components of $\gamma_{i}$ corresponding to shadowed partitions in $\mathcal{P}_2^{1}(i)$.
		
	\end{proof}
	Using Lemma \ref{L:1} and Proposition \ref{P:1}, we prove the following.
	\begin{corollary}\label{C:1} Let $1\leq k \leq n$. For any $i\geq 1$, the last row of $P_i$ is given by 
		\begin{multline*}
		\Big[\frac{(-1)^{n}[i]^n\gamma_i}{L_i^n},\frac{(-1)^{n-1}[i]^{n-1}b^{q^{i-1}}\gamma_{i-1}}{L_i^n},\dots,\\
		\frac{(-1)^{n+1-k}[i]^{n+1-k}\gamma_i}{L_i^n},\frac{(-1)^{n-k}[i]^{n-k}b^{q^{i-1}}\gamma_{i-1}}{L_i^n},\dots,\frac{\gamma_i}{L_i^n}   \Big],
		\end{multline*}
		and the $(2n)$-th row of $P_i$ is given by
		\begin{multline*}
		\Big[\frac{(-1)^{n}[i]^nF_i}{L_i^n},\frac{(-1)^{n-1}[i]^{n-1}b^{q^{i-1}}F_{i-1}}{L_i^n},\dots,\\
		\frac{(-1)^{n+1-k}[i]^{n+1-k}F_i}{L_i^n},\frac{(-1)^{n-k}[i]^{n-k}b^{q^{i-1}}F_{i-1}}{L_i^n},\dots,\frac{F_i}{L_i^{n}}\Big].
		\end{multline*}
	\end{corollary}
	\begin{proof} We do induction on $i$. Note that if $i=1$, then Lemma \ref{L:1} shows that the last row and the $(2n)$-th row of $P_1$ are given by
		
		\[
		\Big[\frac{a}{L_1},\frac{b}{L_1},\dots,\frac{a}{L_1^{n}},\frac{b}{L_1^n},\frac{a}{L_1^{n+1}}   \Big]
		\text{ and }
		\Big[\frac{1}{L_1},0,\dots,\frac{1}{L_i^n},0,\frac{1}{L_1^{n+1}}\Big]
		\]
		respectively. By using \eqref{E:Eq1}, we see that $\gamma_1=a/(\theta-\theta^q)=a/L_1$ which implies that the induction hypothesis holds for $i=1$. Assume that it holds for all $i$. We show that the hypothesis holds for $i+1$. Observe that for any $1\leq k \leq n+1$, using the functional equation \eqref{E:Eq1}, we have
		\begin{equation}\label{E:eq2}
		\begin{split}
		\frac{(-1)^k}{[i+1]^k}\Big(\frac{b^{q^{i-1}}\gamma_{i-1}+a^{q^i}\gamma_i}{L_i^{n}}\Big)&=\frac{(-1)^{n+1-k}[i+1]^{n+1-k}\gamma_{i+1}}{L_{i}^n[i+1]^n(-1)^n}\\
		&=\frac{(-1)^{n+1-k}[i+1]^{n+1-k}\gamma_{i+1}}{L_{i+1}^n}.
		\end{split}
		\end{equation}
		Similarly for $1\leq k \leq n$, we also obtain
		\begin{equation}\label{E:eq3}
		\frac{(-1)^k}{[i+1]^k}\frac{b^{q^{i}}\gamma_i}{L_i^{n}}=\frac{(-1)^{n-k}[i+1]^{n-k}b^{q^{i}}\gamma_{i}}{L_{i}^n[i+1]^n(-1)^n}=\frac{(-1)^{n-k}[i+1]^{n-k}b^{q^{i}}\gamma_{i}}{L_{i+1}^n}.
		\end{equation}
		Thus, using Lemma \ref{L:1}, \eqref{E:eq2} and \eqref{E:eq3}, we obtain that the induction hypothesis holds for the last row. For the $(2n)$-th row, by using the similar calculations above replacing $\gamma_{i-1}$ with $F_{i-1}$ and $\gamma_i$ with $F_i$ and applying Proposition \ref{P:1} we also deduce that the latter  statement of the corollary holds.
	\end{proof} 
	By definition, for $i\geq 1$, we see that $F_i$ is of the form 
	\begin{equation}\label{E:eq4}
	F_i=\frac{a^{q^{x_1}}b^{y_1}}{(-1)^{k_1}[n_{11}]\dots [n_{1k_1}]}+\dots +\frac{a^{q^{x_r}}b^{y_r}}{(-1)^{k_r}[n_{r1}]\dots [n_{rk_r}]}
	\end{equation}
	where $x_j,n_j,y_j\in \mathbb{Z}_{\geq 0}$ for $1\leq j \leq r$. Recall  that $t$ is an independent variable over $\CC_{\infty}$. Then for each $F_i\in K$ of the form \eqref{E:eq4}, we set,
	\[
	\tilde{F}_i(t):=\frac{a^{q^{x_1}}b^{y_1}}{(t-\theta^{q^{n_{11}}})\dots (t-\theta^{q^{n_{1k_1}}})}+\dots +\frac{a^{q^{x_r}}b^{y_r}}{(t-\theta^{q^{n_{r1}}})\dots(t-\theta^{q^{n_{rk_{r}}}})}
	\]
	and observe that $\tilde{F}_i(\theta)=F_i$. Furthermore, we define $\tilde{T}_i(t)$ in a similar way using the definition of $T_i$ in \S1.3 so that $\tilde{T}_i(\theta)=T_i$ and for each $i\in \mathbb{Z}_{\geq 0}$, set $\Upsilon_{i}(t):=a\tilde{F}_i(t)+\tilde{T}_i(t)$. It is now easy to notice that $\Upsilon_i(\theta)=\gamma_{i}$. 
	
	Let $g_1(t),
	\dots,g_{2n+1}(t)$ be elements in $K(t)$. We define the matrices $\partial_{1,t}[g_1(t),
	\dots,g_{2n+1}(t)]$ and $\partial_{2,t}[g_1(t),
	\dots,g_{2n+1}(t)]$ in $\Mat_{2n+1}(K(t))$ by
	\[
	\partial_{1,t}[g_1(t),
	\dots,g_{2n+1}(t)]:=\begin{bmatrix}
	\partial_{t}^{n}(g_1(t))& \dots & \partial_{t}^{n}(g_{2n+1}(t))\\
	0&\dots & 0\\
	\vdots &  & \vdots \\
	\partial_{t}(g_1(t))& \dots & \partial_{t}(g_{2n+1}(t))\\
	0&\dots & 0\\
	g_1(t)&\dots & g_{2n+1}(t)
	\end{bmatrix}
	\]
	and 
	\[
	\partial_{2,t}[g_1(t),
	\dots,g_{2n+1}(t)]:=\begin{bmatrix}
	0&\dots & 0\\
	\partial_{t}^{n-1}(g_1(t))& \dots & \partial_{t}^{n-1}(g_{2n+1}(t))\\
	\vdots &  & \vdots \\
	0&\dots & 0\\
	g_1(t)&\dots & g_{2n+1}(t)\\
	0&\dots & 0
	\end{bmatrix}.
	\]
	For any $a(t)\in K(t)$ we also consider $\tilde{d}_{n}[a(t)]\in \Mat_{2n+1}(K(t))$ given by
	\[
	\tilde{d}_{n}[a(t)]:=\begin{bsmallmatrix}
	a(t)&0&\partial_{t}(a(t))&0&\partial^2_{t}(a(t))&\dots &0&\partial_{t}^{n}(a(t))\\
	&a(t)&0&\partial_{t}(a(t))& & \\
	& & \ddots &\ddots  &\ddots  & \\
	& & & \ddots &\ddots  &\ddots&  & \\
	& & & &\ddots &\ddots  &\ddots  & \\
	& &  & & & a(t)&0&\partial_{t}(a(t))\\
	& &  & & & & a(t)&0\\
	& & & & & & &a(t)
	\end{bsmallmatrix}.
	\]
	Let $\mathbb{L}_0(t):=1$ and for $i\geq 1$, we define the deformation $\mathbb{L}_i(t)$ of elements $L_i$ by 
	\[
	\mathbb{L}_i(t):=(t-\theta^{q^i})\dots(t-\theta^{q})\in K[t].
	\] 
	From the definitions, we have  $\mathbb{L}_i(\theta)=L_i$.  Finally, define $\mathbb{P}_0(t):=\Id_{2n+1}$ and for all $i\geq 1$ and $1\leq k \leq n$, we set  
	\begin{align*}
	\mathbb{P}_i(t):&=\partial_{1,t}\bigg[\frac{(t-\theta^{q^i})^n\Upsilon_{i}(t)}{\mathbb{L}_i(t)^n}, \frac{(t-\theta^{q^i})^{n-1}b^{q^{i-1}}\Upsilon_{i-1}(t)}{\mathbb{L}_i(t)^n},\dots, \\
	&\ \ \ \ \ \ \ \ \ \ \ \  \frac{(t-\theta^{q^i})^{n+1-k}\Upsilon_{i}(t)}{\mathbb{L}_i(t)^n}, \frac{(t-\theta^{q^i})^{n-k}b^{q^{i-1}}\Upsilon_{i-1}(t)}{\mathbb{L}_i(t)^n},\dots,\frac{\Upsilon_{i}(t)}{\mathbb{L}_i(t)^n}\bigg]\\
	&\ \ \ +\partial_{2,t}\bigg[\frac{(t-\theta^{q^i})^n\tilde{F}_{i}(t)}{\mathbb{L}_i(t)^n}, \frac{(t-\theta^{q^i})^{n-1}b^{q^{i-1}}\tilde{F}_{i-1}(t)}{\mathbb{L}_i(t)^n},\dots,\\
	&\ \ \ \ \ \ \ \ \ \ \ \ \ \frac{(t-\theta^{q^i})^{n+1-k}\tilde{F_{i}}(t)}{\mathbb{L}_i(t)^n},\frac{(t-\theta^{q^i})^{n-k}b^{q^{i-1}}\tilde{F}_{i-1}(t)}{\mathbb{L}_i(t)^n},\dots,\frac{\tilde{F_{i}}(t)}{\mathbb{L}_i(t)^n}\bigg].
	\end{align*}
	The next proposition will be useful to deduce some facts about the domain of convergence of $\Log_{G_n}$. 
	\begin{proposition}\label{P:2} For any $i\geq 0$, we have that  $\mathbb{P}_i(\theta)=P_i$.
	\end{proposition}
	\begin{proof}
		Assume first that $i\geq 1$. Using \eqref{E:rule}, we can obtain
		\begin{equation}\label{E:eq5}
		\begin{split}
		&\tilde{d}_n[t-\theta^{q^i}]\mathbb{P}_i(t)\\
		&=\partial_{1,t}\bigg[\frac{(t-\theta^{q^i})^{n+1}\Upsilon_{i}(t)}{\mathbb{L}_i(t)^n}, \frac{(t-\theta^{q^i})^{n}b^{q^{i-1}}\Upsilon_{i-1}(t)}{\mathbb{L}_i(t)^n},\dots,\frac{(t-\theta^{q^i})^{n+2-k}\Upsilon_{i}(t)}{\mathbb{L}_i(t)^n}, \\
		&\ \ \ \ \ \ \ \  \ \ \ \frac{(t-\theta^{q^i})^{n+1-k}b^{q^{i-1}}\Upsilon_{i-1}(t)}{\mathbb{L}_i(t)^n},\dots,\frac{(t-\theta^{q^{i}})\Upsilon_{i}(t)}{\mathbb{L}_i(t)^n}\bigg]\\
		&\ \ \  + \partial_{2,t}\bigg[\frac{(t-\theta^{q^i})^{n+1}\tilde{F}_{i}(t)}{\mathbb{L}_i(t)^n}, \frac{(t-\theta^{q^i})^{n}b^{q^{i-1}}\tilde{F}_{i-1}(t)}{\mathbb{L}_i(t)^n},\dots,\frac{(t-\theta^{q^i})^{n+2-k}\tilde{F_{i}}(t)}{\mathbb{L}_i(t)^n}, \\
		&\ \ \ \ \ \ \ \ \ \  \ \ \ \frac{(t-\theta^{q^i})^{n+1-k}b^{q^{i-1}}\tilde{F}_{i-1}(t)}{\mathbb{L}_i(t)^n},\dots,\frac{(t-\theta^{q^{i}})\tilde{F_{i}}(t)}{\mathbb{L}_i(t)^n}\bigg].	
		\end{split}
		\end{equation}
		Note also that we have
		\begin{equation}\label{E:eq6}
		\begin{split}
		&\mathbb{P}_i(t)N\\
		&=\partial_{1,t}\bigg[0,0,\frac{(t-\theta^{q^i})^n\Upsilon_i(t)}{\mathbb{L}_i(t)^n},\frac{(t-\theta^{q^i})^{n-1}b^{q^{i-1}}\Upsilon_{i-1}(t)}{\mathbb{L}_i(t)^n},\dots,\\
		&\ \ \ \  \ \ \ \ \ \ \  \frac{(t-\theta^{q^i})^{n+1-k}\Upsilon_{i}(t)}{\mathbb{L}_i(t)^n},\frac{(t-\theta^{q^i})^{n-k}b^{q^{i-1}}\Upsilon_{i-1}(t)}{\mathbb{L}_i(t)^n},\dots, \frac{(t-\theta^{q^i})\Upsilon_i(t)}{\mathbb{L}_i(t)^n}\bigg]\\
		&\ \ \ +\partial_{2,t}\bigg[0,0,\frac{(t-\theta^{q^i})^n\tilde{F}_i(t)}{\mathbb{L}_i(t)^n},\frac{(t-\theta^{q^i})^{n-1}b^{q^{i-1}}\tilde{F}_{i-1}(t)}{\mathbb{L}_i(t)^n},\dots,\\
		&\ \ \ \  \ \ \ \ \ \ \ \ \  \frac{(t-\theta^{q^i})^{n+1-k}\tilde{F_{i}}(t)}{\mathbb{L}_i(t)^n},\frac{(t-\theta^{q^i})^{n-k}b^{q^{i-1}}\tilde{F}_{i-1}(t)}{\mathbb{L}_i(t)^n},\dots, \frac{(t-\theta^{q^i})\tilde{F}_i(t)}{\mathbb{L}_i(t)^n}\bigg].
		\end{split}
		\end{equation}
		Thus, combining \eqref{E:eq5} and \eqref{E:eq6} we obtain
		\begin{equation}\label{E:eq16}
		\begin{split}
		&\tilde{d}_n[t-\theta^{q^i}]\mathbb{P}_i(t)-	\mathbb{P}_i(t)N\\
		&=\partial_{1,t}\bigg[\frac{(t-\theta^{q^i})^{n+1}\Upsilon_{i}(t)}{\mathbb{L}_i(t)^n}, \frac{(t-\theta^{q^i})^{n}b^{q^{i-1}}\Upsilon_{i-1}(t)}{\mathbb{L}_i(t)^n},0,\dots,0\bigg]	\\
		&\ \ \ \ \  +\partial_{2,t}\bigg[\frac{(t-\theta^{q^i})^{n+1}\tilde{F}_{i}(t)}{\mathbb{L}_i(t)^n}, \frac{(t-\theta^{q^i})^{n}b^{q^{i-1}}\tilde{F}_{i-1}(t)}{\mathbb{L}_i(t)^n},0,\dots,0\bigg].
		\end{split}
		\end{equation}
		On the other hand, again by using \eqref{E:rule}, we also have
		\begin{equation}\label{E:eq17}
		\begin{split}
		\mathbb{P}_{i-1}E^{(i-1)}&=\partial_{1,t}\bigg[\frac{\Upsilon_{i-2}(t)b^{q^{i-2}}+a^{q^{i-1}}\Upsilon_{i-1}(t)}{\mathbb{L}_{i-1}(t)^n},\frac{b^{q^{i-1}}\Upsilon_{i-1}(t)}{\mathbb{L}_{i-1}(t)^n},0,\dots,0\bigg]\\
		&\ \ \ \  +\partial_{2,t}\bigg[\frac{\tilde{F}_{i-2}(t)b^{q^{i-2}}+a^{q^{i-1}}\tilde{F}_{i-1}(t)}{\mathbb{L}_{i-1}(t)^n},\frac{b^{q^{i-1}}\tilde{F}_{i-1}(t)}{\mathbb{L}_{i-1}(t)^n},0,\dots,0\bigg].
		\end{split}
		\end{equation}
		By the functional equation \eqref{E:Eq1}, we see that 
		\begin{equation}\label{E:eq9}
		\begin{split}
		\frac{\Upsilon_{i-2}(t)b^{q^{i-2}}+a^{q^{i-1}}\Upsilon_{i-1}(t)}{\mathbb{L}_{i-1}(t)^n}\Big|_{t=\theta}&=\frac{(t-\theta^{q^i})^{n+1}\Upsilon_i(t)}{(t-\theta^{q^i})^n\mathbb{L}_{i-1}^n(t)} \Big|_{t=\theta}\\
		&=\frac{(t-\theta^{q^i})^{n+1}\Upsilon_{i}(t)}{\mathbb{L}_i(t)^n}\Big|_{t=\theta}.
		\end{split}
		\end{equation}
		Similarly, we also have
		\begin{equation}\label{E:eq10}
		\frac{b^{q^{i-1}}(t-\theta^{q^i})^n\Upsilon_{i-1}(t)}{(t-\theta^{q^i})^n\mathbb{L}_{i-1}(t)^n}\Big|_{t=\theta}=\frac{b^{q^{i-1}}(t-\theta^{q^i})^n\Upsilon_{i-1}(t)}{\mathbb{L}_{i}(t)^n}\Big|_{t=\theta}.
		\end{equation}
		By Proposition \ref{P:1}, similar calculation as in \eqref{E:eq9} and \eqref{E:eq10}  also gives 
		\begin{equation}\label{E:eq11}
		\begin{split}
		\frac{\tilde{F}_{i-2}(t)b^{q^{i-2}}+a^{q^{i-1}}\tilde{F}_{i-1}(t)}{\mathbb{L}_{i-1}(t)^n}\Big|_{t=\theta}&=\frac{(t-\theta^{q^i})^{n+1}\tilde{F}_i(t)}{(t-\theta^{q^i})^n\mathbb{L}_{i-1}^n(t)} \Big|_{t=\theta}\\
		&=\frac{(t-\theta^{q^i})^{n+1}\tilde{F}_{i}(t)}{\mathbb{L}_i(t)^n}\Big|_{t=\theta}
		\end{split}
		\end{equation}
		and
		\begin{equation}\label{E:eq12}
		\frac{b^{q^{i-1}}(t-\theta^{q^i})^n\tilde{F}_{i-1}(t)}{(t-\theta^{q^i})^n\mathbb{L}_{i-1}(t)^n}\Big|_{t=\theta}=\frac{b^{q^{i-1}}(t-\theta^{q^i})^n\tilde{F}_{i-1}(t)}{\mathbb{L}_{i}(t)^n}\Big|_{t=\theta}.
		\end{equation}
		Moreover, by definition, we have
		\begin{equation}\label{E:eq14}
		\tilde{d}_n[t-\theta^{q^i}]_{|t=\theta}=(\theta-\theta^{q^i})\Id_{2n+1}+N.
		\end{equation}
		Finally, evaluating both sides of \eqref{E:eq16} and \eqref{E:eq17} at $t=\theta$ together with using \eqref{E:eq9}, \eqref{E:eq10}, \eqref{E:eq11}, \eqref{E:eq12} and \eqref{E:eq14}, we see that
		\[
		((\theta-\theta^{q^i})\Id_{2n+1}+N)\mathbb{P}_i(\theta)-\mathbb{P}_i(\theta)N=\mathbb{P}_{i-1}(\theta)E^{(i-1)}
		\]
		which implies that the matrix $\mathbb{P}_i(\theta)$ satisfies the same functional equation \eqref{E:funceq} as $P_i$ does. Since such a solution is unique, we conclude that $\mathbb{P}_i(\theta)=P_i$ for $i\geq 1$. Note that when $i=0$, the proposition follows from the definition of $\mathbb{P}_0(t)$. Thus we finish the proof.
	\end{proof}
	We are now ready to give the main result of this section.
	\begin{theorem}\label{T:T1} Let $\phi$ be the Drinfeld $A$-module of rank 2 defined as in \eqref{E:rank2} such that  $a\in \mathbb{F}_q$ and $b\in \mathbb{F}_q^{\times}$. Let $G_n$ be the $t$-module constructed from $\phi$ and $C^{\otimes n}$ as in \eqref{E:tmodule}. Then the logarithm function $\Log_{G_n}$ of $G_n$ converges on the polydisc $\mathfrak{D}_n:=\{x\in \Lie(G_n)(\CC_{\infty})| \ \ \inorm{x}\leq 1\}$.
	\end{theorem}
	\begin{proof} For $i\geq 1$, we first analyze the last two rows  of $P_i$ by using Corollary \ref{C:1}. Since $a,b\in \mathbb{F}_q$, by \cite[Lem. 4.1]{EP13}, we see that $\inorm{\gamma_{i}}<1$. For any $1\leq k \leq n+1$ we have
		\begin{equation}\label{E:eq18}
		\begin{split}
		\Big|\frac{[i]^{n+1-k}\gamma_{i}}{L_i^n}\Big|_{\infty}<	\Big|\frac{[i]^{n+1-k}}{L_i^n}\Big|_{\infty}&=q^{q^i(n+1-k)-n(q^{i+1}-q)/(q-1)}\\
		&=q^{q^{i}(-k+n+1-nq/(q-1))+nq/(q-1)}.
		\end{split}
		\end{equation}
		Since $b\in \mathbb{F}_q$, for $1\leq k \leq n$, we also have
		\begin{equation}\label{E:eq19}
		\begin{split}
		\Big|\frac{[i]^{n-k}b^{q^{i-1}}\gamma_{i-1}}{L_i^n}\Big|_{\infty}\leq 	\Big|\frac{[i]^{n-k}}{L_i^n}\Big|_{\infty}&=q^{q^i(n-k)-n(q^{i+1}-q)/(q-1)}\\
		&=q^{q^{i}(-k+n-nq/(q-1))+nq/(q-1)}.
		\end{split}
		\end{equation}
		Similar estimation can be also made for the elements $[i]^{n+1-k}F_i/L_i^n$ and $[i]^{n-k}b^{q^{i-1}}F_{i-1}/L_i^n$ by using \eqref{E:eq18} and \eqref{E:eq19} respectively. Thus by Corollary \ref{C:1}, we see that the norm of any element in one of the odd (resp. even) entries in $(2n)$-th or the last row of $P_i$ is bounded by the right hand side of \eqref{E:eq18} (resp. \eqref{E:eq19}). 
		By Proposition \ref{P:2} and the definition of $\mathbb{L}_i(t)$, we see that for $0\leq l,m \leq n$, we have
		\[
		P_{i,(2n+1-2l,2m+1)}=\partial_{t}^{l}\bigg(\frac{\Upsilon_{i}(t)}{(t-\theta^q)^n\dots (t-\theta^{q^{i-1}})^n(t-\theta^{q^{i}})^m}\bigg)\bigg|_{t=\theta},
		\] 
		and for $1\leq j \leq n$, we obtain 
		\[
		P_{i,(2n+1-2l,2j)}=\partial_{t}^{l}\bigg(\frac{b^{q^{i-1}}\Upsilon_{i-1}(t)}{(t-\theta^q)^n\dots (t-\theta^{q^{i-1}})^n(t-\theta^{q^{i}})^j}\bigg)\bigg|_{t=\theta}.
		\]  
		Moreover again by Proposition \ref{P:2}, for $0\leq s \leq n-1$ and $0\leq r \leq n$, we have
		\[
		P_{i,(2n-2s,2r+1)}=\partial_{t}^{s}\bigg(\frac{\tilde{F_{i}}(t)}{(t-\theta^q)^n\dots (t-\theta^{q^{i-1}})^n(t-\theta^{q^{i}})^r}\bigg)\bigg|_{t=\theta},
		\] 
		and for $1\leq j \leq n$, we obtain 
		\[
		P_{i,(2n-2s,2j)}=\partial_{t}^{s}\bigg(\frac{b^{q^{i-1}}\tilde{F}_{i-1}(t)}{(t-\theta^q)^n\dots (t-\theta^{q^{i-1}})^n(t-\theta^{q^{i}})^j}\bigg)\bigg|_{t=\theta}.
		\] 
		Thus, a small calculation implies that the norm of the elements in an odd entry (resp. even) of each row of $P_i$ is smaller than the bound obtained in the right hand side of \eqref{E:eq18} (resp. \eqref{E:eq19}). Now let $x$ be an element in $D$. Then the bound on the norm of $P_i$ implies that 
		\[
		\inorm{P_ix^{q^{i}}}\leq \max_{1\leq k \leq n}q^{q^{i}(-k+n+1-nq/(q-1))+nq/(q-1)}\to 0
		\]
		as $i\to \infty$. Since $x\in \mathfrak{D}_n$ is arbitrary, the function $\Log_{G_n}$ converges on $\mathfrak{D}_n$.
	\end{proof}
	\begin{remark}\label{R:coeff} Let $G_n=(\mathbb{G}_{a/K}^{2n+1},\phi_n)$ be the abelian $t$-module as in the statement of Theorem \ref{T:T1}. For a fixed choice of $(q-1)$-st root of $-b^{-1}$, set $\gamma:=(-b^{-1})^{1/(q-1)}$ and let $\tilde{G}_n=(\mathbb{G}_{a/K}^{2n+1},\tilde{\phi}_n)$ be the  $t$-module given by $\tilde{\phi}_n(\theta)=\gamma^{-1}\phi_n(\theta)\gamma$.  It is easy to check by using the functional equation \eqref{E:funceq} that $\Log_{\tilde{G}_n}=\gamma^{-1}\Log_{G_n}\gamma$. In other words, if $\tilde{P}_i$ is the $i$-th coefficient of $\Log_{\tilde{G}_n}$, then we have $\tilde{P}_i=(-1)^{i}b^{-i}P_i$ for all $i\geq 0$.	Therefore one can also obtain similar results for the logarithm series coefficients of $\tilde{G}_n$ and hence sees that the function  $\Log_{\tilde{G}_n}$ also converges on $\mathfrak{D}_n$.
	\end{remark}

	\section{Class and Unit modules} 
	We fix an abelian $t$-module $G_n$ which is constructed as the tensor product of the $n$-th tensor power of the Carlitz module and a Drinfeld $A$-module $\phi$ of rank 2 given by $\phi_{\theta}=\theta+a\tau +b\tau^2$ such that $a\in \mathbb{F}_q$ and $b\in \mathbb{F}_q^{\times}$. In this section, our aim is to prove some properties of the class and unit module of $G_n$. For more general description of class and unit  modules, we refer the reader to \cite{AnglesTavaresRibeiro}, \cite{Fang} and \cite{Taelman}.  
	
	For $1\leq i \leq 2n+1$, let $e_i\in \Mat_{(2n+1)\times 1}(\mathbb{F}_q)$ be such that the $i$-th coordinate of $e_i$ is 1 and the rest is equal to 0. 
	\begin{lemma}\label{L:l0}
		The $A$-module $\Lie(G_n)(A)$ is free of rank $2n+1$ generated by $e_i$ for $i=1,\dots,2n+1$.
	\end{lemma}
	\begin{proof}
		Suppose that there exist elements $a_1,\dots,a_{2n+1}$ in $A$  such that
		\begin{equation}\label{E:eqproof0}
		\sum_{i\geq 0}\partial_{\phi_n}(a_i) e_i=[0,\dots,0]^{\tr}.
		\end{equation}
		By Lemma \ref{L:dpart}, the equality in \eqref{E:eqproof0} is equivalent to
		\begin{equation}\label{E:prooflinind}
		\sum_{i\geq 0}\partial_{\phi_n}(a_i)e_i=\begin{bmatrix}
		a_1+\partial_{\theta}(a_3)+\dots+\partial_{\theta}^{n}(a_{2n+1})\\
		a_2+\partial_{\theta}(a_4)+\dots+\partial_{\theta}^{n-1}(a_{2n})\\
		\vdots\\
		a_{2n-2}+\partial_{\theta}(a_{2n})\\
		a_{2n-1}+\partial_{\theta}(a_{2n+1})\\
		a_{2n}\\
		a_{2n+1}
		\end{bmatrix}=\begin{bmatrix}
		0\\
		0\\
		\vdots\\
		0\\
		0\\
		0\\
		0\\
		\end{bmatrix}.
		\end{equation}
		Thus we obtain that $a_i=0$ for $1\leq i \leq 2n+1$ recursively. This implies that the set $\{e_1,\dots,e_{2n+1}\}$ is $A$-linearly independent. Using the first equality in \eqref{E:prooflinind}, one can show that the same set also spans the $A$-module $\Lie(G_n)(A)$ and we leave the details to the reader.
	\end{proof}

	Let $v_{\infty}(\cdot)$ be the valuation corresponding to the norm $\inorm{\cdot}$ normalized so that $v_{\infty}(\theta)=-1$. Consider the $\mathbb{F}_q$-module 
	\[
	\mathfrak{m}:=\{x\in \Lie(G_n)(K_{\infty}) \ \ |  v_{\infty}(x)\geq 1\}.
	\] 
	We have the following decomposition of $\mathbb{F}_q$-modules:
	\begin{equation}\label{E:eq1}
	\Lie(G_n)(K_{\infty})=\Lie(G_n)(A)\oplus \mathfrak{m}.
	\end{equation}
	
	Recall that $\Log_{G_n}=\sum_{i\geq 0}P_i\tau^i$ is the logarithm series of $G_n$.
	\begin{proposition}\label{P:P3} For $1\leq i \leq 2n+1$, let $\lambda_i:=\Log_{G_n}(e_i)$. Then the set $\{\lambda_1,\dots,\lambda_{2n+1}\}$ is $A$-linearly independent in $\Lie(G_n)(K_{\infty})$. 
		
	\end{proposition}
	\begin{proof}
		By Theorem \ref{T:T1}, we see that $e_i$ is in the domain of convergence of $\Log_{G_n}$ for $1\leq i \leq 2n+1$. Assume to the contrary that  there exist $a_1,\dots a_{2n+1}\in A$  not all zero satisfying
		\begin{equation}\label{E:eqproof}
		\sum_{i=1}^{2n+1}\partial_{\phi_n}(a_i) \lambda_i=[0,\dots,0]^{\tr}
		\end{equation}
		and let $T:=\max_{1\leq i \leq 2n+1}\{\deg_{\theta}(a_i)\}$. By using Proposition \ref{P:2} and a simple calculation on the valuation of the coefficients of $\Log_{G_n}$, we see that for any $k\geq 1$ and $j\in \{1,\dots,2n+1\}$, $P_k e_j\in \mathfrak{m}$. Since all entries of $P_k$ for $k\geq 1$ has  valuation bigger than 0, if we set $g_i:=\sum_{k=1}^{\infty}P_ke_i$, then we see that $g_i\in \mathfrak{m}$. Now dividing both sides of \eqref{E:eqproof} by $\theta^T$ and using the fact that $P_0=\Id_{2n+1}$, Lemma \ref{L:dpart} yields
		\begin{equation}\label{E2:1}
		\sum_{i=1}^{2n+1}\begin{bsmallmatrix}
		a_{iT}+ a_i^{\prime}&0&\theta^{-T}\partial_{\theta}(a_i)&\dots &0&\theta^{-T}\partial_{\theta}^{n}(a_i)\\
		&a_{iT}+ a_i^{\prime}&0&\theta^{-T}\partial_{\theta}(a_i)& & \\
		& & \ddots &\ddots  &\ddots  & \\
		& &  &  &  & \\
		& &  & a_{iT}+ a_i^{\prime}&0&\theta^{-T}\partial_{\theta}(a_i)\\
		& &  & & a_{iT}+ a_i^{\prime}&0\\
		& & & & & a_{iT}+ a_i^{\prime}
		\end{bsmallmatrix}(e_i+g_i)=\begin{bmatrix}
		0 \\
		\vdots \\
		\vdots \\
		\vdots \\
		0
		\end{bmatrix}
		\end{equation}
		where $a_{iT}$'s are the $\theta^{T}$-th coefficient of $a_i$'s some of which may possibly be zero and $a_i^{\prime}=a_i/\theta^T-a_{iT}$. We also note that  $v_{\infty}(a_i^{\prime})\geq 1$ for $1\leq i \leq 2n+1$. Thus by comparing both sides of \eqref{E2:1}, we see that 
		\[
		\sum_{i=1}^{2n+1}a_{iT}e_i+g_{*}=[0,\dots,0]^{\tr}
		\]
		for some $g_{*}\in \mathfrak{m}$. By the decomposition of $\Lie(G_n)(K_{\infty})$ in \eqref{E:eq1} we have that 
		\[
		\sum_{i=1}^{2n+1}a_{iT}e_i=[0,\dots,0]^{\tr}.
		\]
		But it is only possible if $a_{iT}=0$ for all $i$. Using the similar argument for the remaining coefficients of $a_i$, inductively, we can show that $a_i=0$ for all $i$. But this is a contradiction with the assumption on $a_i$'s. Thus the set $\{\lambda_1,\dots,\lambda_{2n+1}\}$ is $A$-linearly independent.
	\end{proof}
	Let $G$ be an abelian $t$-module and consider the exponential function $\Exp_{G}:\Lie(G)(\CC_{\infty})\to G(\CC_{\infty})$ of $G$. The class module $H(G/A)$ of $G$ is the $A$-module given by the quotient
	\[
	H(G/A):=\frac{G(K_{\infty})}{\Exp_{G}(\Lie(G)(K_{\infty}))+G(A)}.
	\]
	We prove the following proposition.
	\begin{proposition}\label{P:p1} For any $n \geq 1$, we have 
		\[
		H(G_n/A)=\{ 0\}.
		\]
		\begin{proof}
			By Theorem \ref{T:T1}, we have that the set $\mathfrak{m}$ is in the domain of convergence of $\Log_{G_n}$. Since $\Log_{G_n}$ is the formal inverse of $\Exp_{G_n}$, the image of $\Exp_{G_n}$ contains $\mathfrak{m}$. Thus by \eqref{E:eq1} we have that 
			\[
			\Exp_{G_n}(\Lie(G_n)(K_{\infty}))+G_n(A)\supseteq G_n(K_{\infty})
			\]
			which implies that $H(G_n/A)=\{ 0\}$.
		\end{proof}
		
	\end{proposition}

	\begin{definition} Let $V$ be a finite dimensional $K_{\infty}$-vector space. We say that an $A$-module $M\subset V$ is an $A$-lattice in $V$ if it is free and finitely generated over $A$ such that the map $M\otimes_{A}K_{\infty}\to V$ is an isomorphism.
	\end{definition}
	\begin{remark}\label{R:rank} It is important to point out that by \cite[Lem. 1]{AnglesTavaresRibeiro}, an $A$-lattice in $V$ is a free $A$-module of finite rank $r=\dim_{K_{\infty}}(V)$.
		
	\end{remark}

	We now continue with the theory of invertible $A$-lattices introduced in \cite{Deb}.
	\begin{definition}\cite[Def. 2.19]{Deb} \label{D:defff}
		\begin{itemize}
			\item [(i)] An invertible $A$-lattice in $K_{\infty}$ is a tuple $(J,\alpha)$ consisting of a finitely generated and locally free of rank one $A$-module  $J$ and an isomorphism $
			\alpha:J\otimes_{A} K_{\infty}\to K_{\infty}
			$
			of $K_{\infty}$-modules.
			\item[(ii)] Let $\Id_{K_{\infty}}$ be the identity map on $K_{\infty}$. We say $(J_1,\alpha_1)$ and $(J_2,\alpha_2)$ are equivalent whenever there exists an isomorphism $g:J_1\to J_2$ of $A$-modules satisfying 
			\[
			\alpha_2\circ(g\otimes \Id_{K_{\infty}})=\alpha_1
			\]
			where $g\otimes \Id_{K_{\infty}}:J_1\otimes_{A}K_{\infty}\to J_2\otimes_{A}K_{\infty}$ is the map induced by $g$.
		\end{itemize}
	\end{definition}
	One can obtain that the relation between invertible $A$-lattices stated in Definition \ref{D:defff}(ii) is an equivalence relation and we denote the set of equivalence classes of invertible $A$-lattices in $K_{\infty}$ by $\Pic(A,K_{\infty})$.

	Given two finitely generated locally free $A$-modules $\tilde{J}_1$ and $\tilde{J}_2$, we can construct an invertible $A$-lattice $J$ as follows: Let $d_1$ and $d_2$ be the rank of $\tilde{J}_1$ and $\tilde{J}_2$ over $A$ and  $\alpha:\tilde{J}_1\otimes_{A}K_{\infty}\to \tilde{J}_2\otimes_{A}K_{\infty}$ be an isomorphism of $K_{\infty}$-modules. For $i=1,2$, we define $\det_{A}(\tilde{J}_i)$ to be the $d_i$-th exterior power $\wedge^{d_i}\tilde{J}_i$ of $\tilde{J}_i$. Since $A$ is a Dedekind domain, we see that $\Hom_{A}(\det_{A}(\tilde{J}_1),\det_{A}(\tilde{J}_2))$ is a finitely generated and locally free $A$-module of rank one. Moreover the tuple $J:=(\Hom_{A}(\det_{A}(\tilde{J}_1),\det_{A}(\tilde{J}_2)),\tilde{\alpha})$ is an element of $\Pic(A,K_{\infty})$ where \[
	\tilde{\alpha}:\Hom_{A}(\det_{A}(\tilde{J}_1),\det_{A}(\tilde{J}_2))\otimes_{A}K_{\infty}\to K_{\infty}
	\]
	is the isomorphism induced by $\alpha$. 
	\begin{proposition} \cite[Prop. 2.38]{Deb}\label{P2:2}
		There exists a unique homomorphism $v:\Pic(A,K_{\infty})\to \mathbb{Q}$ whose composition with $K_{\infty}^{\times} \to \Pic(A,K_{\infty})$ is the valuation $v_{\infty}$.
	\end{proposition}
	We call an element $g=\sum_{j\leq j_0}c_j\theta^{j}\in K_{\infty}^{\times}$ monic if the leading coefficient $c_{j_0}\in \mathbb{F}_q^{\times}$ is equal to 1. The monic generator of the $A$-module $\Hom_{A}(\det_{A}(\tilde{J}_1),\det_{A}(\tilde{J}_2))$ is denoted by $[\tilde{J}_1:\tilde{J}_2]_{A}$. Note that Proposition \ref{P2:2} actually implies that 
	\[
	v(\Hom_{A}(\det_{A}(\tilde{J}_1),\det_{A}(\tilde{J}_2)),\tilde{\alpha})=v_{\infty}([\tilde{J}_1:\tilde{J}_2]_{A}).
	\]
	
	We continue with an observation due to Angl\`{e}s and Tavares Ribeiro \cite[Sec. 2]{AnglesTavaresRibeiro}. Let $V_1^{\prime}$ and $V_2^{\prime}$ be $A$-lattices in a finite dimensional  $K_{\infty}$-vector subspace $V^{\prime}$ in $V$ defined by $V_i^{\prime}:=V_i\cap V^{\prime}$ for $i=1,2$. Then $V_1/V_1^{\prime}$ and $V_2/V_2^{\prime}$ are $A$-lattices in $V/V^{\prime}$ with the property that 
	\begin{equation}\label{E:proppp}
	\Big[V_1/V_1^{\prime}:V_2/V_2^{\prime}\Big]_{A}=\frac{[V_1:V_2]_{A}}{[V_1^{\prime}:V_2^{\prime}]_{A}}.
	\end{equation}

	Using the homomorphism $v:\Pic(A,K_{\infty})\to \mathbb{Q}$, Debry also proved the following.
	\begin{lemma}\cite[Cor. 2.40]{Deb} \label{P2:P2}
		Let $\Lambda\subset \Lambda^{\prime}$ be two finitely generated locally free $A$-modules of the same rank. Then we have
		\[
		v_{\infty}([\Lambda^{\prime}:\Lambda]_{A})=-\dim_{\mathbb{F}_q}(\Lambda^{\prime}/\Lambda).
		\]
	\end{lemma}
	
	\begin{definition}\label{D:unit} 
		Assume that $G$ is an abelian $t$-module. We define the unit module $U(G/A)$ corresponding to $G$ by
		\[
		U(G/A):=\{x\in \Lie(G)(K_{\infty}) \ \ | \ \ \Exp_{G}(x)\in \Lie(G)(A)\}\subset \Lie(G)(K_{\infty}).
		\]
	\end{definition}
	By \cite[Thm. 1.10]{Fang}, we know that $\Lie(G)(A)$ and $U(G/A)$ are $A$-lattices in $\Lie(G)(K_{\infty})$. Since $A$ is a Dedekind domain, $\Lie(G)(A)$ and $U(G/A)$ are also locally free $A$-modules. 
	
	The following proposition is useful to determine the generators of the unit module $U(G_n/A)$.
	\begin{proposition}[{cf. \cite[Prop. 4.29]{Deb}}]\label{P:p2} Let $\Lambda$ be a finitely generated locally free  $A$-submodule of $U(G/A)$ of the same rank as $\Lie(G)(A)$. 
		If $v_{\infty}([\Lie(G)(A):\Lambda]_{A})=0$ and $H(G/A)=\{ 0\}$, then $\Lambda=U(G/A)$.
	\end{proposition}
	\begin{proof} We first note that the inclusion of $\Lambda$ in $\Lie(G)(K_{\infty})$ induces an isomorphism 
		\[
		\iota:\Lambda \otimes_{A} K_{\infty} \to \Lie(G)(K_{\infty})\otimes_{A} K_{\infty}
		\]
		and hence the tuple  $(\Hom_{A}(\det_{A}(\Lie(G)(A)),\det_{A}(\Lambda)),\tilde{\iota})$ is in $\Pic(A,K_{\infty})$ where 
		\[
		\tilde{\iota}:\Hom_{A}(\det_{A}(\Lie(G)(A)),\det_{A}(\Lambda))\otimes_{A}K_{\infty}\to K_{\infty}
		\]
		is the isomorphism induced by $\iota$.  Recall the definition of $L(G/A)$ from \S2.7. Note that $v_{\infty}(L(G/A))=0$. Thus by the assumption on $H(G/A)$ and \cite[Thm. 1.10]{Fang}, we have $v_{\infty}([\Lie(G)(A):U(G/A)]_{A})=0$. Moreover, by \cite[Lem. 2.23]{Deb}, we have the equality that 
		\begin{equation}\label{E:val}
		[\Lie(G)(A):\Lambda]_{A}=[\Lie(G)(A):U(G/A)]_{A}[U(G/A):\Lambda]_{A}.
		\end{equation}
		Note also that the tuple $(\Hom_{A}(\det_{A}(U(G/A)),\det_{A}(\Lambda)),\tilde{\alpha})$ is an element in $\Pic(A,K_{\infty})$ where 
		\[
		\tilde{\alpha}:\Hom_{A}(\det_{A}(U(G/A)),\det_{A}(\Lambda))\otimes_{A}K_{\infty}\to K_{\infty}
		\]
		is the isomorphism induced by 
		\[
		\alpha:U(G/A)\otimes_{A}K_{\infty}\to \Lambda\otimes_{A}K_{\infty}.
		\]
		Calculating the valuation of both sides of \eqref{E:val}, we see that 
		\[
		v_{\infty}([\Lie(G)(A):\Lambda]_{A})=v_{\infty}([U(G/A):\Lambda]_{A}).
		\]
		Thus by the assumption, we get $v_{\infty}([U(G/A):\Lambda]_{A})=0$. The proposition now follows from Lemma \ref{P2:P2}.
	\end{proof}
	Observe that by  \cite[Thm. 1.10]{Fang} and Proposition \ref{P:p1}, we obtain 
	\begin{equation}\label{E:cnf}
	L(G_n/A)=[\Lie(G_n)(A):U(G_n/A)]_{A}.
	\end{equation}
	
	We set 
	\[
	\Lambda:=\bigoplus_{i=1}^{2n+1} A\cdot  \lambda_i\subset \Lie(G_n)(K_{\infty}),
	\]
	to be the $A$-module generated by $\lambda_i$'s, where $\lambda_i$ is as in Proposition \ref{P:P3} whose $A$-module structure is induced by $\partial_{\phi_n}$. By Proposition \ref{P:P3}, we observe that $\Lambda$ is an $A$-submodule of $U(G_n/A)$ which is free of rank $2n+1$ and therefore has the same rank as $\Lie(G_n)(A)$ by Lemma \ref{L:l0}. It is also locally free as it is a torsion free module over the Dedekind domain $A$. Moreover one can note that the element $[\Lie(G_n)(A):\Lambda]_{A}\in K_{\infty}^{\times}$ can be given by the determinant of the matrix $\Pi\in \Mat_{2n+1}(K_{\infty})$ whose $i$-th column having the coordinates of $\lambda_i=\Log_{G_n}(e_i)$. Since the idea of the proof of Theorem \ref{T:T1} implies that the entries of $P_i$ for $i\geq 1$ has valuation bigger than 0, we obtain  $[\Lie(G_n)(A):\Lambda]_{A}=\det(\Pi)=1+m^{\prime}$ where $m^{\prime}\in \mathfrak{m}$. Therefore  we have $v_{\infty}([\Lie(G_n)(A):\Lambda]_{A})=v_{\infty}(1+m^{\prime})=0$. Thus using Proposition \ref{P:p2}, we conclude the following.
	\begin{theorem}\label{T:unit} 
		\begin{itemize}
			\item[(i)] 	We have
			\[
			U(G_n/A)=\bigoplus_{i=1}^{2n+1} A\cdot \lambda_i\subset\Lie(G_n)(K_{\infty})
			\]
			where $\lambda_i=\Log_{G_n}(e_i)$ for all $1\leq i\leq 2n+1$.
			\item[(ii)] Let $\tilde{G}_n=(\mathbb{G}_{a/K}^{2n+1},\tilde{\phi}_n)$ be the $t$-module defined as in Remark \ref{R:coeff}. Then $\tilde{G}_n$ is an abelian $t$-module. Furthermore we have
			\[
			U(\tilde{G}_n/A)=\bigoplus_{i=1}^{2n+1}A\cdot \tilde{\lambda}_i\subset\Lie(\tilde{G}_n)(K_{\infty})
			\]
			where $\tilde{\lambda}_i=\Log_{\tilde{G}_n}(e_i)$ for all $1\leq i\leq 2n+1$.
		\end{itemize}	
	\end{theorem}
	\begin{proof} The first part follows from the previous discussion. For the second part, we first show that $\tilde{G}_n$ is an abelian $t$-module. Let $M^{\prime}$ be the Taelman $t$-motive defined as $M^{\prime}:=M_\phi\otimes\textbf{C}^{\otimes n+1}\otimes\det(M_\phi)^{\vee}$ where $M_{\phi}$ is the effective $t$-motive corresponding to $\phi$. Note that 
		\[
		\textbf{C}^{\otimes n+1}\otimes\det(M_\phi)^{\vee}=(\textbf{C}^{\otimes n},1)\otimes (-b\textbf{1},-1) = (M_{0},0)
		\]
		where $M_{0}$ is the effective t-motive $K[t]\tilde{m}$ with $K[t]$-basis $\{\tilde{m}\}$ and whose $\tau$-action is given by $\tau \cdot f\tilde{m}=-bf^{(1)}(t-\theta)^{n}\tilde{m}$
		for all $f\in K[t]$. Thus, $M^{\prime}=(M_\phi,0)\otimes (M_{0},0)=(M_\phi\otimes M_{0},0)$ is indeed an effective $t$-motive. By a similar calculation as in \S 2.3, we see that $M'$ is free of rank $2n+1$ over $K[\tau]$ and the $t$-module corresponding to $M^{\prime}$ is given by $\tilde{G}_n=(\mathbb{G}_{a/K}^{2n+1},\tilde{\phi}_n)$ which implies that $\tilde{G}_n$ is an abelian $t$-module. By Remark \ref{R:coeff}, we know that the logarithm coefficients of $\tilde{G}_n$ are $\mathbb{F}_q^{\times}$-multiple of the logarithm coefficients of $G_n$. Therefore one can obtain by using the idea of the proof of Proposition \ref{P:P3} that the set $\{\tilde{\lambda}_1,\dots,\tilde{\lambda}_{2n+1}\}$ is $A$-linearly independent in $\Lie(\tilde{G}_n)(K_{\infty})$. On the other hand, since $\tilde{\phi}_n(\theta)=\gamma^{-1}\phi_n(\theta)\gamma$ where $\gamma=(-b^{-1})^{1/(q-1)}$, by using the same idea in the proof of Lemma \ref{L:l0} that $\Lie(\tilde{G}_n)(A)$ is free of rank $2n+1$ generated by $e_i$ for $i=1,\dots,2n+1$. Now the second part follows by using Proposition \ref{P:p2} and obtaining the formula 	
		\[L(\tilde{G}_n/A)=[\Lie(\tilde{G}_n)(A):U(\tilde{G}_n/A)]_{A}\]
		in a similar way used to deduce \eqref{E:cnf}.
	\end{proof}

	\section{The proof of the main result}
	Throughout this section, we let the $t$-module $G_n=(\mathbb{G}_{a/K}^{2n+1},\phi_n)$ be constructed from $\phi$ given by $\phi_{\theta}=\theta+a\tau+b\tau^2$ such that $a\in \mathbb{F}_q$ and $b\in \mathbb{F}_q^{\times}$ and $C^{\otimes n}$ as in \eqref{E:tmodule}.
	\subsection{The dual $t$-motive of $G_n$}  Before giving the definition of the dual $t$-motive of $G_n$, we require further setup. Assume that $L$ is a perfect field and is an extension of $K$ in $\CC_{\infty}$. We define the non-commutative polynomial ring $L[\sigma]$ with the condition
	\[
	\sigma B=B^{(-1)}\sigma, \ \ B\in L.
	\] 
	
	For $k,m\in \ZZ_{\geq 1}$, let $\Mat_{k\times m}(L)[\tau]$ be the set of polynomials of $\tau$ with coefficients in $\Mat_{k\times m}(L)$. For any $g=g_0+g_1\tau +\dots +g_l\tau^l\in \Mat_{k\times m}(L)[\tau]$, we further define $g^{*}\in \Mat_{m\times k}(L)[\sigma]$ by $g^{*}:=g_0^{\tr}+(g_1^{(-1)})^{\tr}\sigma +\dots +(g_l^{(-l)})^{\tr}\sigma^l$. Moreover we set the ring $L[t,\sigma]:=L[t][\sigma]$ subject to the condition
	\[
	ct=tc, \ \ \sigma c=c^{(-1)}\sigma, \ \ t\sigma=\sigma t, \ \  c\in L.
	\]
	
	\begin{definition}
		\begin{itemize}
			\item[(i)] A dual $t$-motive $H$ over $L$ is a left $L[t,\sigma]$-module which is free and finitely generated over  $L[\sigma]$ satisfying \[
			(t-\theta)^{s}H/\sigma H=\{0\}
			\]
			for some $s\in \mathbb{Z}_{\geq 0}$.
			\item[(ii)] The morphisms between dual $t$-motives are left $L[t,\tau]$-module homomorphisms and we denote the category of dual $t$-motives by $\mathbb{H}$.
			\item[(iii)] We define the dual $t$-motive $H_G$ of a $t$-module $G=(\mathbb{G}_{a/L}^{d},\psi)$ as the left $L[t,\sigma]$-module $H_G:=\Mat_{1\times d}(L)[\sigma]$ whose $L[\sigma]$-module structure given by the free $L[\sigma]$-module $\Mat_{1\times d}(L)[\sigma]$  and the $L[t]$-module action is given by 
			\begin{equation}\label{E:sigmaact2}
			ct^i\cdot h=ch\psi(\theta^i)^{*}, \ \  h\in H_G, \ \  c\in L.
			\end{equation}
		\end{itemize}
	\end{definition}
	In his unpublished notes, Anderson proved that the category $\mathbb{H}$ of Anderson dual $t$-motives over $L$ is equivalent to the category $\mathcal{G}$ of $t$-modules defined over $L$ (see \cite[Sec. 2.5]{HartlJuschka16} and \cite[Sec. 4.4]{BP02} for more details). In the above definition, we see that one can correspond a $t$-module to a dual $t$-motive. We now describe how we can relate a dual $t$-motive to a $t$-module. 
	
	Let $\delta_0,\delta_1:L[\sigma]\to L$ be  $L$-linear homomorphisms defined by 
	\[
	\delta_0\Big(\sum_{i\geq 0}a_i\sigma^i\Big):=a_0
	\text{, and } 
	\delta_1\Big(\sum_{i\geq 0}a_i\sigma^i\Big):=\sum_{i\geq 0}a_i^{q^i}.
	\]
	It is easy to observe that the kernel of $\delta_0$ (resp. $\delta_1$ ) is equal to $\sigma L[\sigma]$ (resp. $(\sigma-1)L[\sigma]$). By a slight abuse of notation, we further denote the maps $\delta_0,\delta_1:\Mat_{1\times d}(L)[\sigma]\to \Mat_{d\times 1}(L)$ given by $\delta_0((f_1,\dots,f_d))=(\delta_0(f_1),\dots,\delta_0(f_d))^{\tr}$ and $\delta_1((f_1,\dots,f_d))=(\delta_1(f_1),\dots,\delta_1(f_d))^{\tr}$ for any $(f_1,\dots,f_d)\in \Mat_{1\times d}(L)[\sigma]$.
	
	Let $H$ be a dual $t$-motive which is free of rank $d$ over $L[\sigma]$. We consider the map $\eta:H\to \Mat_{d\times 1}(L)$ given by the composition of $\delta_1$ with the map which gives the $L[\sigma]$-module isomorphism $H\cong\Mat_{1\times d}(L)[\sigma]$. One can see that $\eta$ induces the isomorphism 
	\begin{equation}\label{E:isom}
	H/(\sigma-1)H\cong L^d.
	\end{equation}
	The $\mathbb{F}_q[t]$-module structure on $H/(\sigma-1)H$ allows us to put an $\mathbb{F}_q[t]$-module structure on $L^d$ by using \eqref{E:isom} which provides us an $\mathbb{F}_q$-linear ring homomorphism $\eta^{\prime}:\mathbb{F}_q[t]\to \Mat_{d}(L)[\tau]$. This process induces a $t$-module $G=(\mathbb{G}_{a/L}^{d},\psi)$ where   $\psi:A\to \Mat_{d}(L)[\tau]$ is given by setting $\psi(\theta):=\eta^{\prime}(t)$, and such $G$ is called the $t$-module corresponding to $H$.
	
	Our aim is to describe the dual $t$-motive corresponding to the $t$-module $G_n=(\mathbb{G}_{a/L},\phi_n)$. Let  $H_n$ be the left $L[t,\sigma]$-module given by  
	\[
	H_n:=L[t]h_1\oplus L[t]h_2
	\]
	with the $L[t]$-basis $\{h_1,h_2\}$ on which  $\sigma$ acts as  
	\begin{equation}\label{E:sigmaact}
	\sigma\cdot h_1=\frac{(t-\theta)^n}{b}h_2 \ \ \ \ \text{and   } \ \ \ \ 
	\sigma\cdot h_2=(t-\theta)^{n+1}h_1-\frac{a(t-\theta)^{n}}{b}h_2.
	\end{equation}
	By the definition of the $\sigma$-action on $H_n$, we easily see that $H_n$ is the  dual $t$-motive with the $L[\sigma]$-basis 
	$
	\{(t-\theta)^{n}h_1,(t-\theta)^{n-1}h_2,\dots,(t-\theta)h_1,h_2,h_1 \}.
	$
	
	We consider an element $f=\sum_{i=0}^na_i(t-\theta)^ih_1+\sum_{i=0}^{n-1}b_i(t-\theta)^{i}h_2\in H_n$ where $a_i,b_i\in L$. Then we compute

	\[
	\begin{split}
	t f&=(t-\theta+\theta)f\\
	&=\sum_{i=1}^na_{i-1}(t-\theta)^ih_1+a_n(t-\theta)^{n+1}h_1  \\
	&\ \ \ \ +\sum_{i=1}^{n-1}b_{i-1}(t-\theta)^ih_2+b_{n-1}(t-\theta)^{n}h_2 +\sum_{i=0}^n\theta a_i(t-\theta)^ih_1\\
	&\ \ \ \ +\sum_{i=0}^{n-1}\theta b_i(t-\theta)^ih_2\\
	&=\sum_{i=1}^na_{i-1}(t-\theta)^ih_1+\sum_{i=1}^{n-1}b_{i-1}(t-\theta)^ih_2+a_n\sigma\cdot h_2+a_na\sigma\cdot h_1\\
	&\ \ \ \ +b_{n-1}b\sigma\cdot h_1+\sum_{i=0}^n\theta a_i(t-\theta)^ih_1+\sum_{i=0}^{n-1}\theta b_i(t-\theta)^{i}h_2\\
	&=\sum_{i=1}^na_{i-1}(t-\theta)^ih_1+\sum_{i=1}^{n-1}b_{i-1}(t-\theta)^ih_2+(\sigma-1)\cdot (a_n^qh_2) + a_n^qh_2\\
	&\ \ \ \ +(\sigma-1)\cdot (a_n^qah_1) +a_n^qah_1+(\sigma-1)\cdot (b_{n-1}^qbh_1)+b_{n-1}^qbh_1\\
	&\ \ \ \ +\sum_{i=0}^n\theta a_i(t-\theta)^ih_1+\sum_{i=0}^{n-1}\theta b_i(t-\theta)^{i}h_2.
	\end{split}
	\]
	Thus we obtain
	\begin{align*}
	t f&=(\theta a_n+a_{n-1})(t-\theta)^{n}h_1+(\theta b_{n-1}+b_{n-2})(t-\theta)^{n-1}h_2\\
	&\ \ \ \  +(\theta a_{n-1}+a_{n-2})(t-\theta)^{n-1}h_1+\dots+(\theta b_{1}+b_{0})(t-\theta)h_2\\
	&\ \ \ \ +(\theta a_{1}+a_{0})(t-\theta)h_1+(\theta b_0+a_n^q)h_2+(\theta a_0+aa_n^q+bb_{n-1}^q)h_1\\
	&\ \ \ \ +(\sigma-1)\cdot (a_{n}^qh_2+a_n^qah_1+b_{n-1}^qbh_1).
	\end{align*}
	Therefore the $\mathbb{F}_q$-linear ring homomorphism $\eta^{\prime}:\mathbb{F}_q[t]\to \Mat_{2n+1}(L)[\tau]$ described above is given by $ \eta^{\prime}(t)=\phi_n(\theta)$ which proves that $G_{n}$ is the $t$-module corresponding to $H_n$ under the equivalence between the  categories $\mathbb{H}$ and $\mathcal{G}$.
	
	For all $i\in \{1,\dots,2n+1\}$, we recall the definition of $e_i$ and then set $e_i^{\vee}:=e_i^{\tr}\in \Mat_{1\times(2n+1) }(\mathbb{F}_q)$. 
	Let $H_{G_n}$ be the dual $t$-motive of $G_n$. We have that $H_n\cong H_{G_n}$ as left $L[t,\sigma]$-modules. We now define a left $L[t,\sigma]$-module isomorphism
	\[
	\iota:H_n\to H_{G_n}
	\]
	by $\iota(h_1)=e^{\vee}_{2n+1}$ and $\iota(h_2)=e^{\vee}_{2n}$. Using the $L[t]$-module action on $H_n$ defined as in \eqref{E:sigmaact} and the $L[t]$-module action on $H_{G_n}$ defined as in \eqref{E:sigmaact2}, we obtain $\iota((t-\theta)^{n-j}h_1)=e^{\vee}_{2j+1}$ for $0\leq j \leq n$ and $\iota((t-\theta)^{n-1-j}h_2)=e^{\vee}_{2(j+1)}$ for $0\leq j \leq n-1$.
	
	Next proposition is crucial to deduce our main result. 
	\begin{proposition}\label{P:delta} Let $f_1h_1+f_2h_2$ be an arbitrary element in $H_n$ for some $f_1,f_2\in L[t]$. Then 
		\begin{multline*}
		\delta_0\circ\iota(f_1h_1+f_2h_2)=[
		\partial_{t}^n(f_1)_{|t=\theta},
		\partial_{t}^{n-1}(f_2)_{|t=\theta},
		\partial_{t}^{n-1}(f_1)_{|t=\theta},
		\dots,\\
		\partial_{t}(f_2)_{|t=\theta},
		\partial_{t}(f_1)_{|t=\theta},
		f_2{_{|t=\theta}},
		f_1{_{|t=\theta}}
		]^{\tr}.
		\end{multline*}
	\end{proposition}
	\begin{proof}For $i=1,2$, we can write $f_i=\sum_{j=0}^{d_i}g_{i,j}(t-\theta)^j$ where $d_i$ is a non-negative integer and $g_{i,j}\in L$. By the $\sigma$-action on $H_n$ given as in \eqref{E:sigmaact}, we see that $(t-\theta)^nh_2\in \sigma H_n$ and $(t-\theta)^{n+1}h_1=\sigma h_2-a\sigma h_1\in \sigma H_n$. Thus we have 
		\[
		\delta_0\circ\iota(f_1h_1+f_2h_2)=\delta_0\circ\iota\Big(\sum_{j=0}^ng_{1,j}(t-\theta)^jh_1+\sum_{j=0}^{n-1}g_{2,j}(t-\theta)^jh_2\Big).
		\]
		Since both $\delta_0$ and $\iota$ are $L$-linear maps, we see that 
		\begin{equation}\label{E:proof1}
		\delta_0\circ\iota(f_1h_1+f_2h_2)=(g_{1,n},g_{2,n-1},g_{1,n-1},\dots,g_{2,0},g_{1,0})^{\tr}.
		\end{equation}
		By Proposition \ref{P:CGM}, we have $\partial_{t}^j(f_i)_{|t=\theta}=g_{i,j}$ for $i=1,2$. Thus by \eqref{E:proof1}, we obtain
		\begin{multline*}
		\delta_0\circ\iota(f_1h_1+f_2h_2)=[\partial_{t}^n(f_1)_{|t=\theta},\partial_{t}^{n-1}(f_2)_{|t=\theta},\partial_{t}^{n-1}(f_1)_{|t=\theta},\dots,\\
		\partial_{t}(f_2)_{|t=\theta},
		\partial_{t}(f_1)_{|t=\theta},f_{2}{_{|t=\theta}},f_{1}{_{|t=\theta}}]^{\tr}
		\end{multline*}
		which finishes the proof.
	\end{proof}
	\subsection{Inverse of the Frobenius} In this subsection, for any $j\in \mathbb{Z}_{\geq 0}$, we introduce a crucial map $\varphi_j:H_n\to \Mat_{(2n+1)\times 1}(L)$ for our purposes (see \cite[Sec. 2]{ADTR} for more details). We set $p_0(t):=1$ and for any $j\in \mathbb{Z}_{\geq 0}$, choose $p_j(t)\in K[t]$ such that 
	\begin{equation}\label{E:modulo}
	p_j(t)(t-\theta^{q^j})^{n+1}\equiv 1 \pmod{(t-\theta)^{n+1}K[t]}.
	\end{equation}
	By the $\sigma$-action on $H_n$ as in \eqref{E:sigmaact}, we see that $(t-\theta)^{n+1}H_n\subset \sigma H_n$. Thus for any $h\in H_n$, there exists a unique element $x\in H_n$ such that $\prod_{k=0}^{j-1}(t-\theta^{q^{-k}})^{n+1}h=\sigma^{j}x$. We now set 
	\[
	\varphi_j(h):=\delta_0\circ\iota(p_0(t)\dots p_j(t)x).
	\]
	The map $\varphi_j$ is called the $j$-th inverse of the Frobenius. 
	
	We define the $K_{\infty}$-vector space $W$ given by
	\[
	W:=\{c_1\cdot\varphi_1(h_1)+c_2\cdot \varphi_1(h_2)\ \ |c_1,c_2\in K_{\infty} \}\subset \Lie(G_n)(\CC_{\infty}).
	\]
	Note that the $K_{\infty}$-vector space structure on $W$ is induced by the map given in \eqref{E:exten}.
	
	We now calculate $\varphi_1(h_1)$ and $\varphi_1(h_2)$. Note that by the property \eqref{E:modulo} of the polynomial $p_1(t)\in K[t]$, we have
	\begin{equation}\label{E:inv1}
	p_1(t)(t-\theta^{q})^{n+1}-r(t)(t-\theta)^{n+1}=1
	\end{equation}
	for some $r(t)\in K[t]$. Hence, writing $p_1(t)=\sum_{i=0}^ny_i(t-\theta)^i+y(t)(t-\theta)^{n+1}$ and $(t-\theta^{q})^{n+1}=(t-\theta+\theta-\theta^q)^{n+1}=\sum_{j=0}^{n}c_j(t-\theta)^j+(t-\theta)^{n+1}$ for some $y(t)\in K[t]$ and $c_j,y_i\in K$ where $1\leq i,j \leq n$, we see that \eqref{E:inv1} becomes
	\begin{equation}\label{E:inv2}
	\Big(\sum_{i=0}^ny_i(t-\theta)^i\Big)\Big(\sum_{j=0}^{n}c_j(t-\theta)^j\Big)+Y(t)(t-\theta)^{n+1}=1
	\end{equation}
	for some $Y(t)\in K[t]$. From \eqref{E:inv2}, one can calculate $y_s$ for $0\leq s \leq n$ recursively using the equations
	\begin{equation}\label{E:inv4}
	\sum_{i=0}^sy_ic_{m-i}=\begin{cases}
	1\text{ if } s=0\\0 \text{ otherwise}
	\end{cases}.
	\end{equation}
	Since equations in \eqref{E:inv4} to determine $y_s$ are also used to determine the coefficient of $(t-\theta)^s$ in the Taylor expansion of $(t-\theta^q)^{-(n+1)}$ at $t=\theta$, we see that 
	\begin{equation}\label{E:inv5}
	p_1(t)(t-\theta^q)^j \pmod{(t-\theta)^{n+1}K[t]}\equiv \frac{1}{(t-\theta^q)^{n+1-j}}\pmod{(t-\theta)^{n+1}K[t]}
	\end{equation}
	for $j=0,1$. Using \eqref{E:sigmaact} and our assumption on the elements $a$ and $b$, we also observe that 
	\[
	(t-\theta)^{n+1}h_2=(t-\theta)\sigma bh_1=\sigma b(t-\theta^q)h_1.
	\]
	Thus by Proposition \ref{P:CGM}, Proposition \ref{P:delta} and  \eqref{E:inv5}, we have
	\[
	\varphi_1(h_2)=\delta_0\circ\iota(p_0(t)p_1(t)b(t-\theta)^qh_1)=\begin{pmatrix}
	\partial_{t}^{n}\Big(\frac{b}{(t-\theta^q)^n}\Big)_{|t=\theta}\\
	0\\
	\partial_{t}^{n-1}\Big(\frac{b}{(t-\theta^q)^n}\Big)_{|t=\theta}\\
	\vdots\\
	0\\
	\frac{b}{(t-\theta^q)^n}_{|t=\theta}
	\end{pmatrix}.
	\]
	On the other hand, by \eqref{E:sigmaact}, we also have
	\[
	(t-\theta)^{n+1}h_1=\sigma h_2+a\sigma h_1=\sigma(h_2+a h_1).
	\]
	Thus similar to the calculation of $\varphi_1(h_2)$, we now obtain
	\[
	\varphi_1(h_1)=\delta_0\circ\iota(p_0(t)p_1(t)(h_2+ah_1))=\begin{pmatrix}
	\partial_{t}^{n}\Big(\frac{a}{(t-\theta^q)^{n+1}}\Big)_{|t=\theta}\\
	\partial_{t}^{n-1}\Big(\frac{1}{(t-\theta^q)^{n+1}}\Big)_{|t=\theta}\\
	\partial_{t}^{n-1}\Big(\frac{a}{(t-\theta^q)^{n+1}}\Big)_{|t=\theta}\\
	\vdots\\
	\frac{1}{(t-\theta^q)^{n+1}}_{|t=\theta}\\
	\frac{a}{(t-\theta^q)^{n+1}}_{|t=\theta}
	\end{pmatrix}.
	\]
	
	\begin{theorem}\label{T:proves} For any positive integer $n$ satisfying $2n+1\leq q$, we have 
		\[ 
		W=K_{\infty}\cdot e_{2n}\oplus K_{\infty}\cdot e_{2n+1}\cong \frac{\Lie(G_n)(K_{\infty})}{(\partial_{\phi_{n}}(\theta)-\theta\Id_{2n+1})\Lie(G_n)(K_{\infty})}.
		\]
		In particular $\varphi_1(h_1)$ and $\varphi_1(h_2)$ are $K_{\infty}$-linearly independent.
	\end{theorem}
	\begin{proof}
		We first prove the inclusion $\supseteq$. To do this we need to find $c_{i,j}\in K_{\infty}$ for $i,j\in \{1,2\}$ so that 
		\[
		\sum_{j=1}^2\partial_{\phi_n}(c_{i,j})\varphi_1(h_j)=e_{2n+2-i}.
		\]
		Note that by the properties of hyperderivatives (see \cite[Lem. 2.3.23]{PLogAlg}) and Proposition \ref{P:CGM}  we have
		\begin{equation}\label{E:same}
		\begin{split}
		(t-\theta^q)^n&=\sum_{i=0}^n\partial_{t}^i((t-\theta^q)^n)_{|t=\theta}(t-\theta)^i\\
		&=\sum_{i=0}^n\binom{n}{i}(\theta-\theta^q)^{n-i}(t-\theta)^i\\
		&=\sum_{i=0}^n\partial_{\theta}^i((\theta-\theta^q)^n)(t-\theta)^i.
		\end{split}
		\end{equation}
		Thus, \eqref{E:same} implies that $\partial_{t}^i((t-\theta^q)^n)_{|t=\theta}=\partial_{\theta}^i((\theta-\theta^q)^n)$ for all $0\leq i \leq n$. Therefore, using \eqref{E:rule}, we see that 
		\begin{equation}\label{E:same2}
		\begin{split}
		&\sum_{\substack{i.j\geq 0\\i+j=m}}\partial_{\theta}^j((\theta-\theta^q)^n)\partial_{t}^i\Big(\frac{1}{(t-\theta^q)^n}\Big)_{|t=\theta}\\
		&\ \ =\sum_{\substack{i.j\geq 0\\i+j=m}}\partial_{t}^j((t-\theta^q)^n)_{|t=\theta}\partial_{t}^i\Big(\frac{1}{(t-\theta^q)^n}\Big)_{|t=\theta}\\
		&\ \ =\partial_{t}^m(1)\\
		&\ \ =0
		\end{split}
		\end{equation}
		for all $1\leq m\leq n$.
		We now choose $c_{1,2}=b^{-1}(\theta-\theta^q)^n$ and $c_{1,1}=0$. Thus by \eqref{E:same2}, one can obtain 
		\begin{equation}\label{E:same4}
		\partial_{\phi_n}(c_{1,2})\varphi_1(h_2)=d_n[c_{1,2}]\varphi_1(h_2)=e_{2n+1}.
		\end{equation}
		Similarly, if we choose $c_{2,2}=-ab^{-1}(\theta-\theta^q)^n$ and $c_{2,1}=(\theta-\theta^q)^{n+1}$, we see that 
		\begin{equation}\label{E:same5}
		\partial_{\phi_n}(c_{2,1})\varphi_1(h_1)+\partial_{\phi_n}(c_{2,2})\varphi_1(h_2)=e_{2n}.
		\end{equation}
		Thus, we have $W\supseteq K_{\infty}\cdot e_{2n}\oplus K_{\infty}\cdot e_{2n+1}$. On the other hand, note that since the matrices $\partial_{\phi_n}(c_{i,j})=d_n[c_{i,j}]$ has non-zero determinant when $c_{i,j}\neq 0$, by using \eqref{E:same4}, we obtain $\varphi_1(h_2)=d_n[c_{1,2}]^{-1}e_{2n+1}$. Thus, by multiplying both sides of \eqref{E:same5} with $d_n[c_{2,1}]^{-1}$, we can obtain $\varphi_1(h_1)$ in terms of a linear combination of $e_{2n}$ and $e_{2n+1}$. Thus we have the desired inclusion $W\subseteq K_{\infty}\cdot e_{2n}\oplus K_{\infty}\cdot e_{2n+1}$. It also implies that $\varphi_1(h_1)$ and $\varphi_1(h_2)$ are $K_{\infty}$-linearly independent. The isomorphism of the $K_{\infty}$-vector spaces in the statement of the theorem follows from the definition of $G_n$ and the details are left to the reader.
	\end{proof}
	

	We recall the $t$-module  $\tilde{G}_n=(\mathbb{G}_{a/K}^{2n+1},\tilde{\phi}_n)$ defined as in Remark \ref{R:coeff}. To prove the main result of the paper, we need some further analysis on $\tilde{G}_n$ and its dual $t$-motive. Consider the left  $L[t,\sigma]$-module
	\[
	\tilde{H}_n:=L[t]\tilde{h}_1\oplus L[t]\tilde{h}_2
	\]
	on which  $\sigma$ acts as  
	\[
	\sigma\cdot \tilde{h}_1=-(t-\theta)^n\tilde{h}_2\ \ \ \ \text{and   } \ \ \ \ 
	\sigma\cdot\tilde{h}_2=-b(t-\theta)^{n+1}\tilde{h}_1+a(t-\theta)^{n}\tilde{h}_2.
	\]
	
	By a straightforward modification of the calculations in the present section, one can obtain that $\tilde{H}_n$ is the dual $t$-motive corresponding to $\tilde{G}_n$. Furthermore, we have $\varphi_1(\tilde{h}_1)=-b^{-1}\varphi_1(h_1)$ and $\varphi_1(\tilde{h}_2)=-b^{-1}\varphi_1(h_2)$. Since $b\in \mathbb{F}_q^{\times}$, by Theorem \ref{T:proves}, one can obtain
	\begin{equation}\label{E:ww}
	\begin{split}
	\tilde{W}:&=\{c_1\cdot\varphi_1(\tilde{h}_1)+c_2\cdot \varphi_1(\tilde{h}_2)\ \ |c_1,c_2\in K_{\infty} \}\\
	&=K_{\infty}\cdot e_{2n}\oplus K_{\infty}\cdot e_{2n+1}\\
	&\cong \frac{\Lie(\tilde{G}_n)(K_{\infty})}{(\partial_{\tilde{\phi}_{n}}(\theta)-\theta\Id_{2n+1})\Lie(\tilde{G}_n)(K_{\infty})}.
	\end{split}
	\end{equation}
	
	By using Theorem \ref{T:proves} and \cite[Prop. 4.2, 4.3, Thm. 4.4]{ADTR}, one can easily deduce the following theorem.
	\begin{theorem}\label{T:module} 
		\begin{itemize}
			\item[(i)] Let $\tilde{W}$ be the $K_{\infty}$-vector space as in \eqref{E:ww} and let $\Log_{\tilde{G}_n}=\sum_{i\geq 0}\tilde{P}_i\tau^i$ be the logarithm series of the $t$-module $\tilde{G}_n$. Then for any natural number $n$ satisfying $2n+1 \leq q$, the $K_{\infty}$-vector space (via the  action of  $\partial_{\tilde{\phi}_{n}}$) generated by $\tilde{P}_i\tau^i(x)$ for all $i\geq 1$ and any $x\in \Lie(\tilde{G}_n)(\CC_{\infty})$ is contained in $\tilde{W}$.  	
			\item[(ii)] $U(\tilde{G}_n/A)\cap \tilde{W}$ and $\Lie(\tilde{G}_n)(A)\cap \tilde{W}$ are $A$-lattices in $\tilde{W}$.
		\end{itemize}
	\end{theorem}
	
	\subsection{Taelman $L$-values and Goss $L$-series}
	For a given abelian $t$-module $G$, we recall the definition of the dual of the Taelman $t$-motive corresponding to $G$ and the Taelman $L$-value $L(G/A)$ from \S2.4 and \S2.5.
	
	We continue with letting $\phi$ to be the Drinfeld $A$-module of rank 2 given by $\phi_{\theta}=\theta+a\tau+b\tau^2$ such that $a\in \mathbb{F}_q$ and $b\in \mathbb{F}_q^{\times}$ and recall that $M_{\phi}$ is the effective $t$-motive corresponding to $\phi$ introduced as in Example \ref{Ex1}(i).  We have by \cite[Rem. 5]{Taelman} (see also \cite[Sec. 2]{ChangEl-GuindyPapanikolas}) that
	\begin{equation}\label{E:dualTael}
	L(M_{\phi}^{\vee},0)=L(\phi/A).
	\end{equation}
	
	Using Goss' results \cite[Sec. 5.6]{Goss} on abelian $t$-modules and applying the theory of Hom shtukas developed in \cite[Sec. 8,12]{M18} on a suitable shtuka model of $\tilde{G}_n$, one can obtain that 
	\[
	L(M_{\tilde{G}_n}^{\vee},0)=L(\tilde{G}_n/A).
	\]
	We also refer the reader to \cite{ADTR20} for another approach using the dual $t$-motive $\tilde{H}_n$ of $\tilde{G}_n$ as well as \cite[Sec. 4.1]{ADTR2020} and the references therein for more details.
	
	\begin{remark}\label{R:conv} 
		Let $\tilde{\phi}$ be the Drinfeld $A$-module  given in Example \ref{Ex:1}(ii) and $M_{\tilde{\phi}}$ be the effective $t$-motive corresponding to $\tilde{\phi}$. Since $L(M_{\tilde{\phi}}^{\vee},s)$ converges for any integer $s\geq 0$ by \cite[Prop. 8]{Taelman}, using \eqref{E:prten11} and \eqref{E:Lser}, we see that  $L(M_{\phi},s)$ converges for any integer $s\geq 1$.
	\end{remark}
	\begin{remark}\label{R:valatone} We briefly explain how one can obtain the value of $L(M_{\phi},n)$ at $n=1$. Let $\tilde{\phi}$ be the Drinfeld $A$-module given in Example \ref{Ex:1}(ii). Consider the $\mathbb{F}_q$-vector space 
		\[
		\mathfrak{m}^{\prime}:=\{x\in K_{\infty} \ \ |  v_{\infty}(x)\geq 1\}.
		\]
		By \cite[Cor. 4.2]{EP13}, we know that $\log_{\tilde{\phi}}$ converges on $\mathfrak{m}^{\prime}$. Thus, in a similar way to the proof of Proposition \ref{P:p1}, we get $H(\tilde{\phi}/A)=\{0\}$. Moreover, again by using \cite[Cor. 4.2]{EP13}, we see that $\log_{\tilde{\phi}}$ converges at 1. This implies that $\log_{\tilde{\phi}}(1)\in U(\tilde{\phi}/A)$. Since $U(\tilde{\phi}/A)$ is an $A$-lattice in $K_{\infty}$, using the minimality of the norm of $\log_{\tilde{\phi}}(1)$ among the elements of $U(\tilde{\phi}/A)$, we see that $U(\tilde{\phi}/A)=\log_{\tilde{\phi}}(1)A$. Thus by \cite[Thm. 1]{Taelman}, we have $L(\tilde{\phi}/A)=\log_{\tilde{\phi}}(1)$. By using Example \ref{Ex:1}(ii), \eqref{E:Lser} and  \eqref{E:dualTael}, we obtain
		\[
		L(M_{\phi},1)=L(M_{\phi}\otimes \textbf{C}^{\vee},0)=L(M_{\tilde{\phi}}^{\vee},0)=L(\tilde{\phi}/A)=\log_{\tilde{\phi}}(1).
		\]
	\end{remark}
	We finish this subsection with the following proposition. 
	\begin{proposition}\label{P:lseries}  For any $n\geq 1$, we have 
		\[
		L(\tilde{G}_n/A)=L(M_\phi,n+1).
		\]
	\end{proposition}
	\begin{proof}  In a similar way to show \eqref{E:prten11}, one can also obtain
		\begin{equation}\label{E:prten2}
		M_{\phi}^{\vee}= M_{\phi}\otimes \det(M_{\phi})^{\vee}.
		\end{equation}
		We also have
		\begin{equation}\label{E:dual}
		\det(M_\phi)^{\vee}\otimes \det(M_\phi)=(-b\textbf{1},-1)\otimes (-b^{-1}\textbf{1},1)\cong\textbf{1}.
		\end{equation}
		Now recall the effective $t$-motive $M^{\prime}=M_\phi\otimes\textbf{C}^{\otimes n+1}\otimes\det(M_\phi)^{\vee}$ from the proof of Theorem \ref{T:unit}(ii) whose corresponding  $t$-module  is given by $\tilde{G}_n=(\mathbb{G}_{a/K}^{2n+1},\tilde{\phi}_n)$. By \eqref{E:can2}, \eqref{E:prten2} and \eqref{E:dual}, we have 
		\[
		\begin{split}
		(M^{\prime})^{{\vee}}= M_\phi\otimes\det(M_\phi)^{\vee}\otimes{(\textbf{C}^{\otimes n+1})}^{\vee}\otimes\det(M_\phi)= M_\phi\otimes{(\textbf{C}^{\otimes n+1})}^{\vee}.
		\end{split}
		\]
		
		Since $\textbf{C}^{\otimes n_1}\otimes \textbf{C}^{\otimes n_2}=\textbf{C}^{\otimes (n_1+n_2)}$ for any positive integers $n_1$ and  $n_2$, using the equality in \eqref{E:Lser} repeatedly together with \eqref{E:dualC}, we obtain
		\[
		L(\tilde{G}_n/A)=L((M^{\prime})^{\vee},0)=L(M_\phi,n+1).
		\]
	\end{proof}

	\subsection{Proof of Theorem \ref{T:specvalue}}
	Using Theorem \ref{T:module}, we prove the next proposition.
	\begin{proposition}\label{P:lattice} For any positive integer $n$ such that $2n+1\leq q$, we have 
		\[U(\tilde{G}_n/A)\cap \tilde{W}=A\cdot  \Log_{\tilde{G}_n}(e_{2n})\oplus A\cdot \Log_{\tilde{G}_n}(e_{2n+1})
		\]
		and 
		\[
		\Lie(\tilde{G}_n)(A)\cap \tilde{W}=A\cdot  e_{2n}\oplus A\cdot  e_{2n+1}.
		\]
	\end{proposition}
	\begin{proof}
		By Theorem \ref{T:unit}(ii), we obtain $U(\tilde{G}_n/A)\cap \tilde{W}\subset \oplus_{i=1}^{2n+1}A\cdot \Log_{\tilde{G}_n}(e_i)$.  By Theorem \ref{T:module}(i), we see that $\tilde{P}_ie_{2n}\in \tilde{W}$ for $i\geq 1$. Since the vectors $e_{2n}$ and $e_{2n+1}$ are in $\tilde{W}$ and $\tilde{W}$ is a finite dimensional normed vector space, the sum $e_{2n+j}+\sum_{i=1}^{\infty}\tilde{P}_ie_{2n+j}=\Log_{\tilde{G}_n}(e_{2n+j})$ is also in $\tilde{W}$ for $j=0,1$. By Theorem \ref{T:module}(ii), we know that $U(\tilde{G}_n/A)\cap \tilde{W}$ is an $A$-lattice in $\tilde{W}$ and therefore is a free $A$-module of rank two by Remark \ref{R:rank}. Thus we obtain $U(\tilde{G}_n/A)\cap \tilde{W}=A\cdot  \Log_{\tilde{G}_n}(e_{2n})\oplus A\cdot  \Log_{\tilde{G}_n}(e_{2n+1})$ as desired. The latter equality in the statement of the proposition can be obtained similarly. 
	\end{proof}

	Now we are ready to state our main result.
	\begin{theorem}\label{T:specvalue} Let $\phi$ be the Drinfeld $A$-module defined as in \eqref{E:rank2} such that $a\in \mathbb{F}_q$ and $b\in \mathbb{F}_q^{\times}$. Then for any positive integer $n$ such that $2n+1 \leq q$, we have
		\begin{multline}\label{E:statement}		
		L(M_\phi,n+1)
		=\bigg(\sum_{i=0}^{\infty}\frac{(-1)^{i}b^{-i}\gamma_i}{L_i^{n}}\bigg)\bigg(1+\sum_{i=1}^{\infty}\frac{(-1)^{i}b^{-(i-1)}F_{i-1}}{L_i^{n}}\bigg)\\
		-\bigg(\sum_{i=1}^{\infty}\frac{(-1)^{i}b^{-(i-1)}\gamma_{i-1}}{L_i^{n}}\bigg)\bigg(\sum_{i=0}^{\infty}\frac{(-1)^{i}b^{-i}F_i}{L_i^{n}}\bigg).
		\end{multline}
		where $F_i$ is the sum of the components of $\gamma_i$ corresponding to  shadowed partitions in $P_{2}^{1}(i)$ for all $i\geq 0$.
	\end{theorem}
	\begin{proof} 
		For $1\leq i \leq 2n+1$, let $\tilde{\lambda}_i=\Log_{\tilde{G}_n}(e_i)$ be as in the statement of Theorem \ref{T:unit}(ii). By \eqref{E:proppp} and Theorem \ref{T:module}(ii), we have 
		\begin{equation}\label{E:PROOF}
		\begin{split}
		[\Lie(\tilde{G}_n)(A)\cap \tilde{W}:U(\tilde{G}_n/A)\cap \tilde{W}]_{A}&=\frac{[\Lie(\tilde{G}_n)(A):U(\tilde{G}_n/A)]_{A}}{\Big[\frac{\Lie(\tilde{G}_n)(A)}{\Lie(\tilde{G}_n)(A)\cap \tilde{W}}:\frac{U(\tilde{G}_n/A)}{U(\tilde{G}_n/A)\cap \tilde{W}}\Big]_{A}}\\
		&=\frac{L(\tilde{G}_n/A)}{\beta}
		\end{split}
		\end{equation}
		where $\beta:=\Big[\frac{\Lie(\tilde{G}_n)(A)}{\Lie(\tilde{G}_n)(A)\cap \tilde{W}}:\frac{U(\tilde{G}_n/A)}{U(\tilde{G}_n/A)\cap \tilde{W}}\Big]_{A}\in K_{\infty}^{\times}$ and the last equality follows from \eqref{E:cnf}.  Theorem \ref{T:module}(i) implies that  $\tilde{P}_ke_i\in \tilde{W}$ for all $k\geq 1$ and $1\leq i \leq 2n+1$. Therefore we have by Theorem \ref{T:unit}(ii) and Proposition \ref{P:lattice} that $\frac{U(\tilde{G}_n/A)}{U(\tilde{G}_n/A)\cap \tilde{W}}\cong \oplus_{i=1}^{2n-1}A\cdot \tilde{\lambda}_i$ is generated by $\tilde{P}_0e_i=e_i$ for $1\leq i \leq 2n-1$ as an $A$-module (via the action of $\partial_{\tilde{\phi}_n}$) in $\frac{\Lie(\tilde{G}_n)(K_{\infty})}{\tilde{W}}$. Using Lemma \ref{L:l0} and Proposition \ref{P:lattice}, one can also obtain that $\frac{\Lie(\tilde{G}_n/A)}{\Lie(\tilde{G}_n/A)\cap \tilde{W}}\cong \oplus_{i=1}^{2n-1}A\cdot e_i$ is generated by $e_i$ for $1\leq i \leq 2n-1$ as an $A$-module (via the action of $\partial_{\tilde{\phi}_n}$) in $\frac{\Lie(\tilde{G}_n)(K_{\infty})}{\tilde{W}}$.
		Thus  we have $\beta=1$. 
		
		For $i,j\in\{0,1\}$,  let $\tilde{\lambda}_{2n+i,2n+j}$ be the $(2n+j)$-th coordinate of the element $\Log_{\tilde{G}_n}(e_{2n+i})$. Consider the matrix
		\[
		\Psi:=\begin{bmatrix}
		\tilde{\lambda}_{2n,2n}&\tilde{\lambda}_{2n+1,2n}\\\tilde{\lambda}_{2n,2n+1}&\tilde{\lambda}_{2n+1,2n+1}
		\end{bmatrix}\in \Mat_2(K_{\infty}).
		\]
		Now by Proposition \ref{P:lattice}, after applying the projection map onto the last two coordinates to the $A$-lattices in the left hand side of \eqref{E:PROOF}, we see that 
		\begin{equation}\label{E:calculation}
		\begin{split}
		&L(\tilde{G}_n/A)\\
		&\ \ \ \ =[\Lie(\tilde{G}_n)(A)\cap \tilde{W}:U(\tilde{G}_n/A)\cap \tilde{W}]_{A}\\
		&\ \ \ \ =\det(\Psi)\\
		&\ \ \ \ =\bigg(\sum_{i=0}^{\infty}\frac{(-1)^{i}b^{-i}\gamma_i}{L_i^{n}}\bigg)\bigg(1+\sum_{i=1}^{\infty}\frac{(-1)^{i}b^{-(i-1)}F_{i-1}}{L_i^{n}}\bigg)\\
		&\ \  \ \ \ \  \ \ \  \ \ \  -\bigg(\sum_{i=1}^{\infty}\frac{(-1)^{i}b^{-(i-1)}\gamma_{i-1}}{L_i^{n}}\bigg)\bigg(\sum_{i=0}^{\infty}\frac{(-1)^{i}b^{-i}F_i}{L_i^{n}}\bigg)
		\end{split}
		\end{equation}
		where the last equality follows from combining Corollary \ref{C:1} with the fact that $\tilde{P}_i=(-1)^{i}b^{-i}P_i$ for $i\geq 0$ and $P_i$ is the $i$-th coefficient of the logarithm series of  $G_n=(\mathbb{G}_{a/K}^{2n+1},\phi_n)$. Note that by Proposition \ref{P:lseries}, we have 
		\[
		L(\tilde{G}_n/A)=L(M_\phi,n+1).
		\]
		Thus the result follows from \eqref{E:calculation}.
	\end{proof}

\end{document}